\documentclass[draftclsnofoot,onecolumn,12pt]{IEEEtran}

\usepackage{amsthm,amsmath,amssymb,enumerate,subfigure,multirow,hhline,color}
\usepackage{xcolor}
\usepackage{hyperref}
\usepackage{graphicx}
\usepackage{graphics}
\usepackage[sort]{cite}

\usepackage{algorithm}
\usepackage{algpseudocode}

\usepackage{epsfig,psfrag}
\usepackage{tabularx}
\usepackage{tabulary}

\usepackage[sort]{cite}

\usepackage{hyperref}
\hypersetup{breaklinks=true,colorlinks=true,linkcolor=black,citecolor=black}

\usepackage[noabbrev,capitalise,nameinlink]{cleveref}

\DeclareMathOperator{\argmax}{arg\,max}

\newtheorem{theorem}{Theorem}
\newtheorem{example}{Example}
\newtheorem{assumption}{Assumption}
\newtheorem{remark}{Remark}

\newtheorem{lemma}{Lemma}
\newtheorem{definition}{Definition}
\newtheorem{proposition}{Proposition}

\allowdisplaybreaks

\title{A Discrete-Time Switching System Analysis of Q-learning}

\author{Donghwan Lee, Jianghai Hu,  Niao He
\thanks{D. Lee is with the Department of Electrical and Engineering, Korea Advanced Institute of Science and Technology (KAIST), Daejeon, 34141, South Korea {\tt\small
donghwan@kaist.ac.kr}.}%
\thanks{J. Hu is with the Department of Electrical and Computer Engineering,
Purdue University, West Lafayette, IN 47906, USA {\tt\small
jianghai@purdue.edu}.}
\thanks{N. He is with the Department of Computer Science,
ETH Z\"{u}rich, 8092 Z\"{u}rich, Switzerland {\tt\small
niao.he@inf.ethz.ch}.}
}

\begin{document}

\maketitle

%=============================================================================================================
\begin{abstract}
This paper develops a novel control-theoretic framework to analyze the non-asymptotic convergence of Q-learning. We show  that the dynamics of asynchronous Q-learning with a constant step-size can be naturally formulated as a discrete-time stochastic affine switching system. Moreover, the evolution of the Q-learning estimation error is over- and underestimated by trajectories of two simpler dynamical systems. Based on these two systems, we derive a new finite-time error bound of asynchronous Q-learning when a constant stepsize is used. Our analysis also sheds light on the overestimation phenomenon of Q-learning. We further illustrate and validate the analysis through numerical simulations.
% It offers novel intuitive insights on analysis of Q-learning mainly based on control theoretic frameworks.
% By filling the gap between both domains in a synergistic way, this approach can potentially facilitate further progress in each field.

\end{abstract}

\section{Introduction}
Q-learning, first introduced by Watkins~\cite{watkins1992q}, is one of the most fundamental and important reinforcement learning algorithms. The theoretical behavior of Q-learning has been extensively studied over the years. Classical analysis of Q-learning mostly focused on asymptotic convergence of asynchronous Q-learning~\cite{tsitsiklis1994asynchronous,jaakkola1994convergence} and synchronous Q-learning~\cite{borkar2000ode}. Substantial advances have been made recently in the guarantee of their finite-time convergence; see \cite{szepesvari1998asymptotic,kearns1999finite,even2003learning, azar2011speedy,beck2012error,wainwright2019stochastic,qu2020finite,li2020sample,chen2021lyapunov}.

To list a few,  Szepesv\'{a}ri in~\cite{szepesvari1998asymptotic} gave the first non-asymptotic  analysis of asynchronous Q-learning under an i.i.d. sampling setting.  ~\cite{even2003learning} first provided the non-asymptotic analysis for both synchronous and asynchronous Q-learning with polynomial and linear step-sizes under a single trajectory Markovian sampling setting.
% \cite{azar2011speedy} proposed a variant of synchronous Q-learning called speedy Q-learning by adding a momentum term, and obtained an accelerated learning rate. Many advances have been made recently in finite-time analysis.
Recently, \cite{wainwright2019stochastic} established the best known bound for synchronous Q-learning under a rescaled linear step-size. In a subsequent work, ~\cite{qu2020finite}  derived a matching bound for asynchronous Q-learning under the Markovian setting using a similar decaying step-size. The sample complexity is further improved with a refined analysis based on constant step-size in \cite{li2020sample} and \cite{chen2021lyapunov}.

Existing results for the most part treat the Q-learning dynamics as a special case of general nonlinear stochastic approximation schemes with Markovian noise. In a different line of work, \cite{lee2020unified} discovered the close connection between Q-learning and continuous-time switching systems. The switching system perspective captures unique features of Q-learning dynamics and  encapsulates a wide spectrum of Q-learning algorithms including asynchronous Q-learning, averaging Q-learning~\cite{lee2020unified}, and Q-learning with function approximation, etc. However, existing O.D.E. analysis of such continuous-time switching systems  yields only asymptotic convergences of Q-learning algorithms and requires diminishing step-sizes. Obtaining a finite-time convergence analysis would require a departure of the switching systems from continuous-time domain to discrete-time domain, which remains an open and challenging question.

In this paper, we aim to close this gap and provide a new finite-time error bound of Q-learning through the lens of discrete-time switching systems. In particular, we focus on asynchronous Q-learning with constant step-sizes for solving a discounted Markov decision process with finite state and action spaces. We first show that asynchronous Q-learning with a constant step-size can be naturally formulated as a stochastic discrete-time affine switching system. This allows us to transform the convergence analysis into a stability analysis of the switching system. However, its stability analysis is nontrivial due to the presence of the affine term and the noise term. The main breakthrough in our analysis lies in developing upper and lower comparison systems whose trajectories over- and underestimate the original system's trajectory. The lower comparison system is a stochastic linear system, while the upper comparison system is a stochastic linear switching system~\cite{lin2009stability}, both of which have much simpler structure than the original system or general nonlinear systems. Our finite-time error bound of Q-learning follows immediately by combining the error bounds of the stochastic linear system (i.e., lower comparison system, which has no affine term) and the error system (i.e., difference of the two comparison systems, which has no noise term). Comparing to existing analysis based on nonlinear stochastic approximation schemes, our analysis seems more intuitive and builds on simple systems. It also sheds new light on the overestimation phenomenon in Q-learning due to the maximization bias~\cite{hasselt2010double}.

Lastly, we emphasize that our goal is to provide new insights and analysis framework to lay out a strong theoretical foundation for Q-learning via its unique connection to discrete-time switching systems, rather than improving existing convergence rates. In particular, as opposed to classical ordinary differential equation analysis/stochastic approximation approaches, the proposed strategy adopts the idea of formulating Q-learning algorithm as a stochastic affine switching system, and directly conduct analysis in discrete-time, which is new in the literature. The switching system model of Q-learning in this paper allows us to already well-established tools in control theory such as Lyapunov analysis, which make the analysis easier and more familiar to researchers in control community. Therefore, we expect that, such control-theoretic analysis could promote more research activities of people with control backgrounds for reinforcement learning, further stimulate the synergy between control theory and reinforcement learning, and open up opportunities to the design of new reinforcement learning algorithms and refined analysis for Q-learning such as double Q-learning~\cite{hasselt2010double}, distributed Q-learning~\cite{kar2013cal}, and speedy Q-learning~\cite{azar2011speedy}.

Moreover, the proposed analysis follows a particularly clean and simple strategy. The core idea that leads to the simplicity is identifying two simpler dynamical systems: `lower comparison system' that is a stochastic linear system and `upper comparison system' that is a stochastic switching system, which have favorable structures that are easily understood and analyzed via control theory: stability of a linear system can be used to derive a finite-time error bound for Q-learning. Therefore, we believe that the convergence analysis of Q-learning can become more familiar to more people including researchers in control theory. Overall, we view our analysis technique as a complement rather than a replacement of existing techniques for Q-learning analysis. Moreover, our approach based on the comparison systems could be of independent interest to the finite-time stability analysis of more general switching systems.

The overall paper consists of the following parts:~\cref{sec:preliminaries} provides preliminary discussions including basics of Markov decision process, switching system, Q-learning, and useful definitions and notations used throughout the paper.~\cref{sec:convergence} provides the main results of the paper, including the switched system models of Q-learning, upper and lower comparison systems, and the finite-time error bounds. We conclude in ~\cref{sec:conclusion} with a discussion on potential extensions of this work.
%%%%%%%%%%%%%%%%%%%%%%%%%%%%%%%%%%%%%%%%%%%%%%%%%%%%%%%%%%%%%%%%%%%%%%%%%%%%%%%%
\section{Preliminaries}\label{sec:preliminaries}

\subsection{Markov decision problem}
We consider the infinite-horizon discounted Markov decision problem (MDP), where the agent sequentially takes actions to maximize cumulative discounted rewards. In a Markov decision process with the state-space ${\cal S}:=\{ 1,2,\ldots ,|{\cal S}|\}$ and action-space ${\cal A}:= \{1,2,\ldots,|{\cal A}|\}$, the decision maker selects an action $a \in {\cal A}$ with the current state $s$, then the state
transits to a state $s'$ with probability $P(s'|s,a)$, and the transition incurs a
reward $r(s,a,s')$.
% , where $P(s,a,s')$ is the state transition probability. from the current state
% $s\in {\cal S}$ to the next state $s' \in {\cal S}$ under action
% $a \in {\cal A}$, and $r(s,a,s')$ is the reward function, which can be also a random variable conditioned on $(s,a,s')$, but
For convenience, we consider a deterministic reward function and simply write
$
r(s_k,a_k ,s_{k + 1}) =:r_k,\quad k \in \{ 0,1,\ldots \}.
$

A deterministic policy, $\pi :{\cal S} \to {\cal A}$, maps a state $s \in {\cal S}$ to an action $\pi(s)\in {\cal A}$. The objective of the Markov decision problem (MDP) is to find a deterministic optimal policy, $\pi^*$, such that the cumulative discounted rewards over infinite time horizons is
maximized, i.e.,
\begin{align*}
%=======================================================================================
\pi^*:= \argmax_{\pi\in \Theta} {\mathbb E}\left[\left.\sum_{k=0}^\infty {\gamma^k r_k}\right|\pi\right],
%=======================================================================================
\end{align*}
where $\gamma \in [0,1)$ is the discount factor, $\Theta$ is the set of all admissible deterministic policies, $(s_0,a_0,s_1,a_1,\ldots)$ is a state-action trajectory generated by the Markov chain under policy $\pi$, and ${\mathbb E}[\cdot|\pi]$ is an expectation conditioned on the policy $\pi$. The Q-function under policy $\pi$ is defined as
\begin{align*}
&Q^{\pi}(s,a)={\mathbb E}\left[ \left. \sum_{k=0}^\infty {\gamma^k r_k} \right|s_0=s,a_0=a,\pi \right],\quad s\in {\cal S},a\in {\cal A},
\end{align*}
and the optimal Q-function is defined as $Q^*(s,a)=Q^{\pi^*}(s,a)$ for all $s\in {\cal S},a\in {\cal A}$. Once $Q^*$ is known, then an optimal policy can be retrieved by the greedy policy $\pi^*(s)=\argmax_{a\in {\cal A}}Q^*(s,a)$. Throughout, we assume that the MDP is ergodic so that the stationary state distribution exists and the Markov decision problem is well posed.

\subsection{Switching system}

Since the switching system is a special form of nonlinear systems, we first consider the nonlinear system
\begin{align}
x_{k+1}=f(x_k),\quad x_0=z \in {\mathbb R}^n,\quad k\in \{1,2,\ldots \},\label{eq:nonlinear-system}
\end{align}
where $x_k\in {\mathbb R}^n$ is the state and $f:{\mathbb R}^n \to {\mathbb R}^n$ is a nonlinear mapping. An important concept in dealing with the nonlinear system is the equilibrium point. A point $x=x^*$ in the state-space is said to be an equilibrium point of~\eqref{eq:nonlinear-system} if it has the property that whenever the state of the system starts at $x^*$, it will remain at $x^*$~\cite{khalil2002nonlinear}. For~\eqref{eq:nonlinear-system}, the equilibrium points are the real roots of the equation $f(x) = x$. The equilibrium point $x^*$ is said to be globally asymptotically stable if for any initial state $x_0 \in {\mathbb R}^n$, $x_k \to x^*$ as $k \to \infty$.

Next, let us consider the particular nonlinear system, the \emph{linear switching system},
\begin{align}
&x_{k+1}=A_{\sigma_k} x_k,\quad x_0=z\in {\mathbb
R}^n,\quad k\in \{0,1,\ldots \},\label{eq:switched-system}
\end{align}
where $x_k \in {\mathbb R}^n$ is the state, $\sigma\in {\mathcal M}:=\{1,2,\ldots,M\}$ is called the mode, $\sigma_k \in
{\mathcal M}$ is called the switching signal, and $\{A_\sigma,\sigma\in {\mathcal M}\}$ are called the subsystem matrices. The switching signal can be either arbitrary or controlled by the user under a certain switching policy. Especially, a state-feedback switching policy is denoted by $\sigma_k = \sigma(x_k)$. A more general class of systems is the {\em affine switching system}
\begin{align*}
&x_{k+1}=A_{\sigma_k} x_k + b_{\sigma_k},\quad x_0=z\in {\mathbb
R}^n,\quad k\in \{0,1,\ldots \},
\end{align*}
where $b_{\sigma_k} \in {\mathbb R}^n$ is the additional input vector, which also switches according to $\sigma_k$. Due to the additional input $b_{\sigma_k}$, its stabilization becomes much more challenging.

\subsection{Revisit Q-learning}
We now briefly review the standard Q-learning and its convergence. Recall that the standard Q-learning updates
\begin{equation*}
\begin{split}
&Q_{k+1}(s_k,a_k)=Q_k(s_k,a_k)\\
&+ \alpha_k(s_k,a_k)\left\{r_k +\gamma \max_{u \in {\cal A}} Q_k (s_{k+1},u)-Q_k(s_k,a_k)\right\},
\end{split}
\end{equation*}
where  $0 \le \alpha_k(s,a) \le 1$ is called the learning rate or step-size associated with the state-action pair $(s,a)$ at iteration $k$. This value is assumed to be zero if
$(s,a)\ne (s_k,a_k)$. If
\begin{align*}
&\sum_{k=0}^\infty{\alpha_k(s,a)}= \infty ,\quad \sum_{k=0}^\infty {\alpha_k^2 (s,a)}<\infty,
\end{align*}
and every state-action pair is visited infinitely often, then the iterate is guaranteed to converge to $Q^*$ with probability one~\cite{sutton1998reinforcement}. Note that the state-action can be visited arbitrarily, which is more general than stochastic visiting rules.

In this paper, we focus on the following setting: $\{(s_k,a_k)\}_{k=0}^{\infty}$ is an i.i.d. samples under a behavior policy $\beta$, where the behavior policy is the policy by which the reinforcement learning agent actually behaves to collect experiences. For simplicity, we assume that the state at each time is sampled from the state distribution $p$, and in this case, the state-action distribution at each time is identically given by
\[
d(s,a) = p (s)\beta (a|s),\quad (s,a) \in {\cal S} \times {\cal A}.
\]

\subsection{Assumptions and definitions}

Throughout, we make the following standard assumptions.
\begin{assumption}\label{assumption:positive-distribution}
%=======================================================================================
$d(s,a)> 0$ holds for all $s\in {\cal S},a \in {\cal A}$.
%=======================================================================================
\end{assumption}
\begin{assumption}\label{assumption:step-size}
The step-size is a constant $\alpha \in (0,1)$.
\end{assumption}
\begin{assumption}\label{assumption:bounded-reward} The reward is bounded as follows:
\begin{align*}
\max _{(s,a,s') \in {\cal S} \times {\cal A} \times {\cal S}} |r (s,a,s')| =:R_{\max}\leq 1.
\end{align*}
\end{assumption}
\begin{assumption}\label{assumption:bounded-Q0} The initial iterate $Q_0$ satisfies $\|Q_0\|_\infty \le 1$.
\end{assumption}
\begin{remark}
All the assumptions are standard and widely used in the reinforcement learning literature. All these assumptions will be used throughout this paper for the convergence proofs. \cref{assumption:positive-distribution} guarantees that every state-action pair is visited infinitely often with probability one for sufficient exploration. This assumption corresponds to the sufficient exploration condition in the standard Q-learning analysis~\cite{jaakkola1994convergence}: every state-action pair $(s,a)$ is visited infinitely often. Moreover, this assumption is used when the state-action visit distribution is given. It has been also considered in~\cite{li2020sample} and~\cite{chen2021lyapunov}. The work in~\cite{beck2012error} considers another exploration condition, called the cover time condition, which states that there is a certain time period, within which all the state-action pair is expected to be visited at least once. Slightly different cover time conditions have been used in~\cite{even2003learning} and~\cite{li2020sample} for convergence rate analysis.
\cref{assumption:bounded-reward} is required to ensure the boundedness of Q-learning iterates, which is applied in almost all reinforcement learning algorithms. The unit bounds imposed on $R_{\max}$ and $Q_0$ are just for simplicity of analysis. The constant step-size in~\cref{assumption:step-size} has been also studied in~\cite{beck2012error} and~~\cite{chen2021lyapunov} using different approaches.
\end{remark}

The following quantities will be frequently used in this paper; hence, we define them for convenience.
\begin{definition}
\begin{enumerate}
\item Maximum state-action visit probability:
\[
d_{\max} := \max_{(s,a)\in {\cal S} \times {\cal A}} d(s,a) \in (0,1).
\]

\item Minimum state-action visit probability:
\[
d_{\min}:= \min_{(s,a) \in {\cal S} \times {\cal A}} d(s,a) \in (0,1).
\]

\item Exponential decay rate:
\begin{align}\label{eq:rho}
    \rho:=1 - \alpha d_{\min} (1-\gamma).
\end{align}
Under~\cref{assumption:step-size}, the decay rate satisfies $\rho \in (0,1)$.
\end{enumerate}
\end{definition}

Throughout the paper, we will use the following compact notations for dynamical system representations:
\begin{align}
P:=& \begin{bmatrix}
   P_1\\
   \vdots\\
   P_{|{\cal A}|}\\
\end{bmatrix},\; R:= \begin{bmatrix}
   R_1 \\
   \vdots \\
   R_{|{\cal A}|} \\
\end{bmatrix},
\; Q:= \begin{bmatrix}
   Q(\cdot,1)\\
  \vdots\\
   Q(\cdot,|{\cal A}|)\\
\end{bmatrix},\nonumber\\
D_a:=& \begin{bmatrix}
   d(1,a) & & \\
   & \ddots & \\
   & & d(|{\cal S}|,a)\\
\end{bmatrix},
\; D:= \begin{bmatrix}
   D_1 & & \\
    & \ddots  & \\
    & & D_{|{\cal A}|} \\
\end{bmatrix},\label{eq:8}
\end{align}
where $P_a = P(\cdot,a,\cdot)\in {\mathbb R}^{|{\cal S}| \times |{\cal S}|}$, $Q(\cdot,a)\in {\mathbb R}^{|{\cal S}|},a\in {\cal A}$ and $R_a(s):={\mathbb E}[r(s,a,s')|s,a]$.
Note that $P\in{\mathbb R}^{|{\cal S}||{\cal A}| \times |{\cal S}|  }$, $R \in {\mathbb R}^{|{\cal S}||{\cal A}|}$, $Q\in {\mathbb R}^{|{\cal S}||{\cal A}|}$, and $D\in {\mathbb R}^{|{\cal S}||{\cal A}| \times |{\cal S}||{\cal A}|}$.
In this notation, the Q-function is encoded as a single vector $Q \in {\mathbb R}^{|{\cal S}||{\cal A}|}$, which enumerates $Q(s,a)$ for all $s \in {\cal S}$ and $a \in {\cal A}$. The single value $Q(s,a)$ can be written as
$
Q(s,a) = (e_a  \otimes e_s )^T Q,
$
where $e_s \in {\mathbb R}^{|{\cal S}|}$ and $e_a \in {\mathbb R}^{|{\cal A}|}$ are $s$-th basis vector (all components are $0$ except for the $s$-th component which is $1$) and $a$-th basis vector, respectively. Note also that under~\cref{assumption:positive-distribution}, $D$ is a nonsingular diagonal matrix with strictly positive diagonal elements.

For any stochastic policy, $\pi:{\cal S}\to \Delta_{|{\cal A}|}$, where $\Delta_{|{\cal A}|}$ is the set of all probability distributions over ${\cal A}$, we define the corresponding action transition matrix as
\begin{align}
\Pi^\pi:=\begin{bmatrix}
   \pi(1)^T \otimes e_1^T\\
   \pi(2)^T \otimes e_2^T\\
    \vdots\\
   \pi(|S|)^T \otimes e_{|{\cal S}|}^T \\
\end{bmatrix}\in {\mathbb R}^{|{\cal S}| \times |{\cal S}||{\cal A}|},\label{eq:swtiching-matrix}
\end{align}
where $e_s \in {\mathbb R}^{|{\cal S}|}$.
Then, it is well known that
$
P\Pi^\pi \in {\mathbb R}^{|{\cal S}||{\cal A}| \times |{\cal S}||{\cal A}|}
$
is the transition probability matrix of the state-action pair under policy $\pi$.
If we consider a deterministic policy, $\pi:{\cal S}\to {\cal A}$, the stochastic policy can be replaced with the corresponding one-hot encoding vector
$
\vec{\pi}(s):=e_{\pi(s)}\in \Delta_{|{\cal A}|},
$
where $e_a \in {\mathbb R}^{|{\cal A}|}$, and the corresponding action transition matrix is identical to~\eqref{eq:swtiching-matrix} with $\pi$ replaced with $\vec{\pi}$. For any given $Q \in {\mathbb R}^{|{\cal S}||{\cal A}|}$, denote the greedy policy w.r.t. $Q$ as
\begin{align}
\pi_Q(s):=\argmax_{a\in {\cal A}} Q(s,a)\in {\cal A}.\label{eq:greedy}
\end{align}

We will use the following shorthand frequently:
\begin{align*}
\Pi_Q:=\Pi^{\pi_Q}.
\end{align*}
We note that this notation, $\Pi_Q:=\Pi^{\pi_Q}$, will play an important role in the derivation of the switching system model in this paper. In particular, the matrix appears in the system parameters, and switches as the greedy policy $\pi_Q(s):=\argmax_{a\in {\cal A}} Q(s,a)\in {\cal A}$ is changed according to $Q$.

The boundedness of Q-learning iterates~\cite{gosavi2006boundedness} plays an important role in our analysis.
\begin{lemma}[Boundedness of Q-learning iterates~\cite{gosavi2006boundedness}]\label{lemma:bounded-Q}
If the step-size is less than one, then for all $k \ge 0$,
\begin{align*}
\|Q_k\|_\infty \le Q_{\max}:= \frac{\max \{R_{\max},\max_{(s,a)\in {\cal S} \times {\cal A}} Q_0 (s,a)\}}{1-\gamma}.
\end{align*}
From~\cref{assumption:bounded-reward} and~\cref{assumption:bounded-Q0}, we can easily see that $Q_{\max}\leq\frac{1}{1-\gamma}$.
\end{lemma}

%%%%%%%%%%%%%%%%%%%%%%%%%%%%%%%%%%%%%%%%%%%%%%%%%%%%%%%%%%%%%%%%%%%%%%%%%%%%%%%%
\section{Finite-time Analysis of Q-learning from Switching System Theory}\label{sec:convergence}

In this section, we study a discrete-time switching system model of Q-learning and establish its finite-time convergence based on the stability analysis of switching system.
We consider a version of Q-learning given in~\cref{algo:standard-Q-learning2}. Compared to the original Q-learning, the step-size $\alpha$ does not depend on the state-action pair and is constant in this paper. Moreover, the output of~\cref{algo:standard-Q-learning2} is the average $\tilde Q_{k}=\frac{1}{k}\sum_{i=0}^{k-1} {Q_k},k\geq 1$ with  instead of the final iteration $Q_k$.
\begin{algorithm}[t]
\caption{Q-Learning with a constant step-size}
  \begin{algorithmic}[1]
    \State Initialize $Q_0 \in {\mathbb R}^{|{\cal S}||{\cal A}|}$ randomly such that $\|Q_0\|_\infty \le 1$.
    \State Set $\tilde Q_0  = Q_0$
    \State Sample $s_0\sim p$
    \For{iteration $k=0,1,\ldots$}
    	\State Sample $a_k\sim \beta(\cdot|s_k)$ and $s_k\sim p(\cdot)$
        \State Sample $s_k'\sim P(s_k,a_k,\cdot)$ and $r_k= r(s_k,a_k,s_k')$
        \State Update $Q_{k+1}(s_k,a_k)=Q_k(s_k,a_k)+\alpha \{r_k+\gamma\max_{u \in {\cal A}} Q_k(s_k',u)-Q_k (s_k,a_k)\}$
        \State Update $\tilde Q_{k+1}  = \tilde Q_k  + \frac{1}{{k + 1}}(Q_{k}  - \tilde Q_k )$
    \EndFor
  \end{algorithmic}\label{algo:standard-Q-learning2}
\end{algorithm}

\subsection{Q-learning as a stochastic affine switching system}
Using the notation introduced, the update in~\cref{algo:standard-Q-learning2} can be rewritten as
\begin{align}
Q_{k+1}=Q_k+\alpha \{DR+\gamma DP\Pi_{Q_k}Q_k-DQ_k +w_{k}\},\label{eq:1}
\end{align}
where
\begin{align}
w_{k}=&(e_{a_k}\otimes e_{s_k} ) r_k+\gamma (e_{a_k}\otimes e_{s_k} )(e_{s_k'})^T \Pi_{Q_k}Q_k\nonumber\\
&-(e_{a_k}\otimes e_{s_k})(e_{a_k}\otimes e_{s_k})^T Q_k -(DR+\gamma DP\Pi_{Q_k}Q_k-DQ_k),\label{eq:w}
\end{align}
and $(s_k,a_k,r_k,s_k')$ is the sample in the $k$-th time-step.
\begin{remark}
Note that in~\cref{algo:standard-Q-learning2}, $(s_k,a_k,s_k')$ is sampled from the joint distribution
\begin{align*}
P({s_{k}'}|{s_k},{a_k})p({s_k})\beta ({a_k}|{s_k}) = P({s_k'}|{s_k},{a_k})d({s_k},{a_k}),
\end{align*}
which is represented by the matrix multiplication, $DP$, in~\eqref{eq:1}. By the definition of matrix $D$ in~\eqref{eq:8}, it is a diagonal matrix whose diagonal entries are an enumeration of $d(s,a) = p(s)\beta (a|s),(s,a) \in {\cal S} \times {\cal A}$. Therefore, it is easy to see that an entry of $DP$ is a joint distribution of a certain $(s,a,s')\in {\cal S} \times {\cal A} \times {\cal S}$. Moreover, from the definition of matrix $\Pi^{\pi}$ in~\eqref{eq:swtiching-matrix} and the greedy policy in~\eqref{eq:greedy}, the multiplication $\Pi_{Q_k}Q_k$ in~\eqref{eq:w} represents that max operator in the Q-function update in~\eqref{algo:standard-Q-learning2}.

In more details, a vector form of the Q-function update in~\eqref{algo:standard-Q-learning2} can be written as
\begin{align}
{Q_{k + 1}} = ({e_{{a_k}}} \otimes {e_{{s_k}}}){({e_{{a_k}}} \otimes {e_{{s_k}}})^T}{Q_k} + \alpha (({e_{{a_k}}} \otimes {e_{{s_k}}}){r_k} + \gamma ({e_{{a_k}}} \otimes {e_{{s_k}}}){({e_{{s_{k'}}}})^T}{\Pi _{{Q_k}}}{Q_k}).\label{eq:10}
\end{align}
Taking the conditional expectation conditioned on $Q_k$ leads to the mean dynamic
\begin{align}
{\mathbb E}[{Q_{k + 1}}|{Q_k}] = D{Q_k} + \alpha (DR + \gamma DP{\Pi _{{Q_k}}}{Q_k}),\label{eq:9}
\end{align}
where $D = {\mathbb E}[({e_{{a_k}}} \otimes {e_{{s_k}}}){({e_{{a_k}}} \otimes {e_{{s_k}}})^T}|{Q_k}]$ and $DP = {\mathbb E}[({e_{{a_k}}} \otimes {e_{{s_k}}}){({e_{{s_{k'}}}})^T}|{Q_k}]$. Adding the right-hand side of~\eqref{eq:9} to~\eqref{eq:10} and subtracting it from~\eqref{eq:10}, we obtain~\eqref{eq:1}.
\end{remark}

Moreover, by definition, the noise term has a zero mean conditioned on $Q_k$, i.e., ${\mathbb E}[w_k|Q_k]=0$.
Recall the definitions $\pi_Q(s)$ and $\Pi_Q$. Invoking the optimal Bellman equation $(\gamma DP\Pi_{Q^*}-D)Q^*+DR=0$,~\eqref{eq:1} can be further rewritten by
\begin{align}
(Q_{k+1} - Q^*) =& \{ I + \alpha (\gamma DP\Pi _{Q_k} - D)\}(Q_k - Q^*)+ \alpha \gamma DP(\Pi_{Q_k} - \Pi_{Q^*})Q^* + \alpha w_k.\label{eq:Q-learning-stochastic-recursion-form}
\end{align}
which is a linear switching system with an extra affine term, $\gamma DP(\Pi_{Q_k} - \Pi_{Q^*})Q^*$, and stochastic noise. For any $Q \in {\mathbb R}^{|{\cal S}||{\cal A}|}$, define
\begin{align*}
A_Q :=I + \alpha(\gamma DP\Pi_Q-D),\quad b_Q:= \alpha \gamma DP(\Pi_{Q} - \Pi_{Q^*})Q^*.
\end{align*}
Using the notation, the Q-learning iteration can be concisely represented as the \emph{stochastic affine switching system}
\begin{align}
Q_{k + 1}- Q^* = A_{Q_k} (Q_k - Q^*) + b_{Q_k} + \alpha w_k,\label{eq:swithcing-system-form}
\end{align}
where $A_{Q_k}$ and $b_{Q_k}$ switch among matrices from $\{I + \alpha (\gamma DP\Pi^\pi-D):\pi\in \Theta\}$ and vectors from $\{\alpha \gamma DP(\Pi^\pi - \Pi^{\pi^*})Q^* :\pi\in\Theta \}$. Note that in the switching system in~\eqref{eq:swithcing-system-form}, the switching signal is not arbitrary, and switching signal follows a switching rule associated with the greedy policy $\pi_{Q_k}(s):=\argmax_{a\in {\cal A}} Q_k(s,a)\in {\cal A}$, which changes according to $Q_k$.

Therefore, the convergence of Q-learning is now reduced to analyzing the stability of the above switching system.
A main obstacle in proving the stability  arises from the presence of the affine and stochastic terms. Without these terms, we can easily establish the exponential stability of the corresponding deterministic switching system, under arbitrary switching policy. Specifically, we have the following result.
\begin{proposition}\label{prop:stability}
For arbitrary $H_k \in {\mathbb R}^{|{\cal S}||{\cal A}|}, k\ge 0$, the linear switching system
\begin{align*}
Q_{k+1} - Q^* &= A_{H_k} (Q_k - Q^*),\quad  Q_0 - Q^*\in {\mathbb R}^{|{\cal S}||{\cal A}|},
\end{align*}
is exponentially stable with
\[
\|Q_k- Q^*\|_\infty\le \rho ^k \|Q_0 - Q^*\|_\infty,\quad k \ge 0,
\]
where $\rho$ is defined in \eqref{eq:rho}.
\end{proposition}

The above result follows immediately from the key fact that $\|A_Q \|_\infty \le \rho$, which we formally state in the lemma below.

\begin{lemma}\label{lemma:max-norm-system-matrix}
For any $Q \in {\mathbb R}^{|{\cal S}||{\cal A}|}$,
$$\|A_Q \|_\infty \le \rho.$$ Here the matrix norm  $\| A \|_\infty :=\max_{1\le i \le m} \sum_{j=1}^n {|A_{ij} |}$ and $A_{ij}$ is the element of $A$ in $i$-th row and $j$-th column.
\end{lemma}
\begin{proof}
Note
% \begin{align*}
% &\|A_Q\|_\infty\\
%  =& \max_{i \in \{1,2,\ldots ,|{\cal S}||{\cal A}|\} } \sum_{j \in \{1,2,\ldots ,|{\cal S}||{\cal A}|\}} {| [I - \alpha D + \alpha \gamma DP\Pi _Q ]_{ij}|}
% \end{align*}
\begin{align*}
\sum_j|[A_Q]_{ij}|=&\sum_j {| [I - \alpha D + \alpha \gamma DP\Pi _Q ]_{ij}|}\\
=&  [I-\alpha D]_{ii} + \sum_j {[\alpha\gamma DP\Pi_Q ]_{ij}}\\
% =& 1 - \alpha [D]_{ii} + \alpha \gamma \sum_j {[DP\Pi _Q ]_{ij} }\\
=& 1 - \alpha [D]_{ii} + \alpha \gamma [D]_{ii} \sum_j {[P\Pi _Q ]_{ij}}\\
=& 1 - \alpha [D]_{ii} + \alpha \gamma [D]_{ii}\\
=& 1 + \alpha [D]_{ii}(\gamma  - 1),
\end{align*}
where the first line is due to the fact that $A_Q$ is a positive matrix. Taking the maximum over $i$, we have
\begin{align*}
\| A_Q \|_\infty =& \max_{i\in \{ 1,2,\ldots ,|{\cal S}||{\cal A}|\} } \{ 1 + \alpha [D]_{ii} (\gamma-1)\}\\
=& 1 - \alpha \min_{(s,a) \in {\cal S} \times {\cal A}} d(s,a)(1 - \gamma ) \\
=& \rho,
\end{align*}
which completes the proof.
\end{proof}
% \begin{proposition}\label{prop:stability}
% The switching system
% \begin{align*}
% Q_{k+1} - Q^* &= \underbrace {\{I + \alpha (\gamma DP\Pi_{H_k} - D)\}}_{=:A_{H_k}} (Q_k - Q^*),\\
% & Q_0 - Q^*\in {\mathbb R}^{|{\cal S}||{\cal A}|},
% \end{align*}
% is exponentially stable under arbitrary $H_k \in {\mathbb R}^{|{\cal S}||{\cal A}|}, k\ge 0$.
% \end{proposition}
% \begin{proof}
% We consider the $\infty$-norm as a Lyapunov function and
% \begin{align*}
% \|Q_{k + 1}- Q^*\|_\infty =& \|A_{H_k}(Q_k- Q^*)\|_\infty\\
% \le & \|A_{H_k}\|_\infty \|Q_k  - Q^*\|_\infty \\
% \le & \rho \|Q_k  - Q^*\|_\infty,\quad k \ge 0.
% \end{align*}

% Therefore, iterating the above inequality yields
% \[
% \|Q_k- Q^*\|_\infty\le \rho ^k \|Q_0 - Q^*\|_\infty,\quad k \ge 0
% \]
% proving the exponential stability of the switching system.
% \end{proof}

However, because of the additional affine term and stochastic noises in the original switching system~\eqref{eq:swithcing-system-form}, it is not obvious how to directly derive its finite-time convergence. To circumvent the difficulty with the affine term, we will resort to two simpler comparison systems, whose trajectories upper and lower bound that of the original system, and can be more easily analyzed. These systems will be called the upper and lower comparison systems depicted in~\cref{fig:1}, which capture important behaviors of Q-learning. The upper comparison system, denoted by $Q_k^U$, upper bounds Q-learning iterate $Q_k$, while the lower comparison system, denoted by $Q_k^L$, lower bounds $Q_k$. The construction of these comparison systems is partly inspired by~\cite{lee2020unified} and exploits the special structure of the Q-learning algorithm. Unlike~\cite{lee2020unified}, here we focus on the discrete-time domain and a finite-time analysis. To address the difficulty with the stochastic noise, we introduce a two-phase analysis: the first phase captures the noise effect of the lower comparison system, while the second phase captures the difference between the two comparison systems when noise effect vanishes.
% The main idea of our analysis lies in using special structures of Q-learning algorithm to over and under bound the original Q-learning iterates.
\begin{figure}[t]
%=============================================================================================================
\centering\epsfig{figure=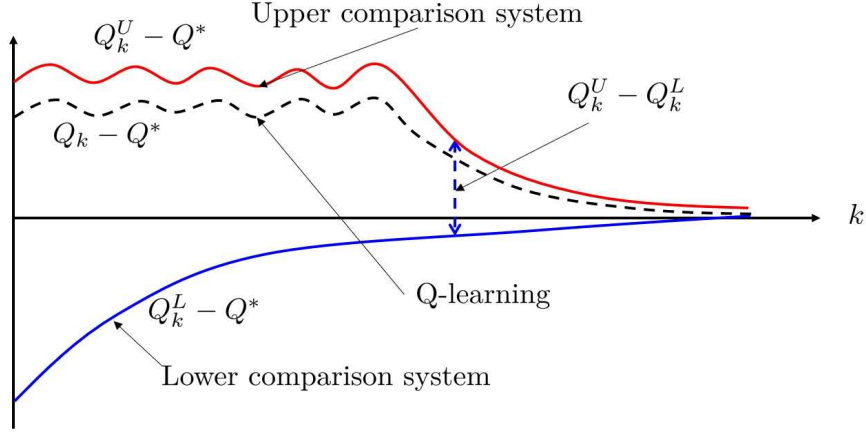,width=12cm,height=6cm}
\caption{Overview of the proposed analysis}\label{fig:1}
%=============================================================================================================
\end{figure}

\subsection{Lower comparison system}
Consider the stochastic linear system
\begin{align}
Q_{k+1}^L- Q^*=A_{Q^*} (Q_k^L - Q^*) + \alpha w_{k},\quad  Q_{0}^L-Q^* \in {\mathbb R}^{|{\cal S}||{\cal A}|},\label{eq:lower-system}
\end{align}
 where the stochastic noise $w_{k}$ is the same as the original system~\eqref{eq:Q-learning-stochastic-recursion-form}. We call it the \emph{lower comparison system}.
%  Its main property is that if $Q_0^L  - Q^*\le Q_0-Q^*$ initially, then $Q_k^L - Q^*\le Q_k-Q^*$ for all $k \ge 0$.
\begin{proposition}\label{prop:lower-bound2}
Suppose $Q_0^L- Q^*\le Q_0 - Q^*$, where $\le$ is used as the element-wise inequality. Then,
$$Q_k^L- Q^*\le Q_k-Q^*,$$
for all $k\geq0$.
\end{proposition}
\begin{proof}
The proof is done by an induction argument. Suppose the result holds for some $k \ge 0$. Then,
\begin{align*}
&(Q_{k+1}- Q^*)\\
% =& (Q_k-Q^*)\\
% &+ \alpha D\{\gamma P\Pi_{Q_k} Q_k - \gamma P\Pi_{Q^*} Q^* - Q_k + Q^*\}+ \alpha w_k\\
=&A_{Q^*} (Q_k-Q^*)+(A_{Q_k}-A_{Q^*}) (Q_k-Q^*)+b_{Q_k}+ \alpha w_k\\
 =&A_{Q^*} (Q_k-Q^*) + \alpha\gamma DP(\Pi_{Q_k}-\Pi_{Q^*})Q_k+\alpha w_k\\
\ge & A_{Q^*} (Q_k-Q^*) + \alpha w_k\\
\ge&  A_{Q^*} (Q_k^L-Q^*) + \alpha w_k\\
=& Q_{k+1}^L-Q^*,
\end{align*}
where the third line is due to $DP(\Pi_{Q_k}- \Pi_{Q^*})Q_k\ge DP(\Pi_{Q^*}-\Pi_{Q^*})Q_k=0$ and the fourth line is due to the hypothesis $Q_k^L- Q^*\le Q_k-Q^*$ and the fact that $A_{Q^*}$ is a positive matrix (all elements are nonnegative). The proof is completed by induction.
\end{proof}
\begin{remark}
Rearranging terms, the original system~\eqref{eq:swithcing-system-form} can be written as
\[
Q_{k + 1}  - Q^*  = A_{Q^*} (Q_k - Q^* ) + \underbrace {\alpha \gamma DP(\Pi _{Q_k }  - \Pi _{Q^* } )Q_k }_{ = :h_{Q_k}} + \alpha w_k
\]
where one can easily prove that $h_{Q_k} \geq 0$ using the definition of $\Pi _{Q_k }$, i.e., $\Pi _{Q_k } Q_k \geq \Pi _{Q^*} Q_k$. Intuitively, removing this nonnegative bias term, $h_{Q_k}$, leads to the lower comparison system. The vector $h_{Q_k}$ represents a portion of the gap between the original and lower systems incurred at a single time-step.
\end{remark}

Note that the mean dynamics of the lower comparison system is simply a linear system. By~\cref{prop:stability}, we have the exponential stability of the mean dynamics:
\begin{align}
\|{\mathbb E}[Q_k^L]-Q^*\|_\infty\le\rho^k \|Q_0^L- Q^*\|_\infty,\quad \forall k \ge 0.\label{eq:2}
\end{align}
% {\color{blue} The main difficulty in further progressing the bound in~\cref{prop:convergence-lower-system} comes from the fact that it is hard to switch the expectation and norm while preserving the inequality. To overcome this obstacle, we will use tools from control theory.} From~\cref{prop:convergence-lower-system},
Furthermore, we can conclude that $A_{Q^*}$ is Schur, i.e., the magnitude of all its engenvalues is strictly less than one, and from the Lyapunov theory for linear systems, there exists a positive definite matrix $M\succ 0$ and $\beta \in (0,1)$ such that
\[
A_{Q^*}^T M A_{Q^*} \preceq \beta M.
\]
The parameter $\beta \in (0,1)$ determines the convergence speed of the state to the origin, and it depends on the structure of the matrix $A_{Q^*}$. We prove that in our case, an upper bound on $\beta$ can be expressed in terms of $\rho$. In fact, we can set $\beta = (\rho+\epsilon)^2$ for arbitrary $\epsilon >0$ such that $\beta\in (0,1)$.
\begin{proposition}\label{prop:Lyapunov-theorem}
For any $\epsilon >0$ such that $\rho + \epsilon \in (0,1)$, there exists the corresponding positive definite $M\succ 0 $ such that
\[
A_{Q^*}^T M A_{Q^*} =  (\rho + \epsilon)^2 (M - I),
\]
and
\begin{align*}
\lambda_{\min}(M) \ge 1,\; \lambda_{\max}(M) \le \frac{|{\cal S}||{\cal A}|}{1-\left({\frac{\rho }{\rho+\epsilon} }\right)^2}.
\end{align*}
\end{proposition}

The above result can be easily verified by setting
$$M = \sum_{k=0}^\infty {\left({\frac{1}{\rho+\epsilon} }\right)^{2k} (A_{Q^*}^k)^T A_{Q^*}^k}.$$
We defer the detailed proof in the Appendix~\ref{app:prop-Lyapunov}. Based on this result, we can derive a finite-time error bound for the lower comparison system.

\begin{theorem}\label{thm:bound-lower-comparison}
Under Assumptions 1-4, for any $N\geq0$, it holds
\begin{align}
{\mathbb E}\left[ {\left\| {\frac{1}{N}\sum\limits_{k = 0}^{N - 1} {Q_k^L }  - Q^* } \right\|_\infty  } \right] \le \sqrt {\frac{{32\alpha |{\cal S}|^2 |{\cal A}|^2 }}{{d_{\min } (1 - \gamma )^3 }} + \frac{1}{N}\frac{{2|{\cal S}|^2 |{\cal A}|^2 }}{{\alpha d_{\min } (1 - \gamma )}}{\mathbb E}[\left\| {Q_0  - Q^* } \right\|_\infty ^2 ]}.\label{eq:lower-system-bound}
\end{align}
\end{theorem}

\begin{proof}
Define the Lyapunov function $V(x) = x^T M x$,
$${\cal F}_k:=\{Q_0^L,Q_0,w_0,Q_1^L,Q_1,w_1,\ldots,Q_{k-1}^L,Q_{k-1},w_{k-1},Q_k^L,Q_k\},$$
and denote $A=A_{Q^*}$. We have
\begin{align*}
&{\mathbb E}[V(Q_{k+1}^L-Q^*)|{\cal F}_k]\\
=& {\mathbb E}[(A(Q_k^L-Q^*) + \alpha w_k )^T M(A(Q_k^L - Q^*)+\alpha w_k)|{\cal F}_k]\\
=& {\mathbb E}[(Q_k^L-Q^*)^T A^T M A(Q_k^L-Q^*)] + \alpha^2 w_k^T M w_k|{\cal F}_k]\\
\leq &(\rho + \epsilon)^2 V(Q_k^L-Q^*)- (\rho + \epsilon)^2 \|Q_k^L  - Q^*\|^2+\lambda _{\max } (M)\alpha^2{\mathbb E}[w_k^T w_k|{\cal F}_k].
\end{align*}
Here $\epsilon>0$ is such that $\rho+\epsilon<1$. The first inequality comes from \cref{prop:Lyapunov-theorem}. The second equality is due to the fact that
\begin{align*}
&{\mathbb E}[{w_k}|{\cal F}_k]\\
=& {\mathbb E}[({e_{{a_k}}} \otimes {e_{{s_k}}}){r_k} + \gamma ({e_{{a_k}}} \otimes {e_{{s_k}}}){({e_{{s_{k'}}}})^T}{\Pi _{{Q_k}}}{Q_k}\\
&- ({e_{{a_k}}} \otimes {e_{{s_k}}}){({e_{{a_k}}} \otimes {e_{{s_k}}})^T}{Q_k} - (DR + \gamma DP{\Pi _{{Q_k}}}{Q_k} - D{Q_k})|{\cal F}_k]\\
=& {\mathbb E}[({e_{{a_k}}} \otimes {e_{{s_k}}}){r_k} + \gamma ({e_{{a_k}}} \otimes {e_{{s_k}}}){({e_{{s_{k'}}}})^T}{\Pi _{{Q_k}}}{Q_k} - ({e_{{a_k}}} \otimes {e_{{s_k}}}){({e_{{a_k}}} \otimes {e_{{s_k}}})^T}{Q_k}|{Q_k}]\\
& - (DR + \gamma DP{\Pi _{{Q_k}}}{Q_k} - D{Q_k})\\
=& DR + \gamma DP{\Pi _{{Q_k}}}{Q_k} - D{Q_k} - (DR + \gamma DP{\Pi _{{Q_k}}}{Q_k} - D{Q_k})\\
=& 0.
\end{align*}
Therefore, we have
\begin{align*}
{\mathbb E}[\alpha w_k^TMA(Q_k^L - {Q^*})|{\cal F}_k] =& {\mathbb E}[\alpha w_k^TMA(Q_k^L - {Q^*})|Q_k^L,{Q_k}]\\
 =& \alpha {\mathbb E}[w_k^T|{Q_k}]MA(Q_k^L - {Q^*})\\
  =& 0
\end{align*}
Subtracting $V(Q_k^L-Q^*)$ from both sides and using $\lambda_{\min}(M) \ge 1$ in~\cref{prop:Lyapunov-theorem} leads to
\begin{align*}
&{\mathbb E}[V(Q_{k+1}^L-Q^*)|{\cal F}_k ] - V(Q_k^L  - Q^* )\\
\le& (\rho+\varepsilon)^2V(Q_k^L-Q^*) - V(Q_k^L-Q^*)- (\rho+\varepsilon)^2 \|Q_k^L  - Q^*\|^2\\
& +\alpha^2 \lambda_{\max} (M){\mathbb E}[w_k^T w_k |{\cal F}_k ]\\
=& ((\rho+\varepsilon)^2-1) V(Q_k^L-Q^*) -(\rho+\varepsilon)^2 \|Q_k^L - Q^*\|^2 \\
& + \alpha^2 \lambda_{\max}(M){\mathbb E}[w_k^T w_k |{\cal F}_k]\\
\le& ((\rho+\varepsilon)^2-1) \|Q_k^L - Q^*\|^2 -(\rho+\varepsilon)^2 \|Q_k^L - Q^*\|^2 \\
& + \alpha^2 \lambda_{\max}(M){\mathbb E}[w_k^T w_k |{\cal F}_k]\\
=& -\|Q_k^L -Q^* \|^2 +\alpha^2 \lambda_{\max} (P){\mathbb E}[w_k^T w_k |{\cal F}_k],
\end{align*}
where the last inequality uses the facts, $(\rho + \varepsilon)^2 -1 < 0$ and $\lambda_{\min}(M) \ge 1$. Therefore, we have
\begin{align*}
&{\mathbb E}[V(Q_{k+1}^L-Q^*)|{\cal F}_k]-V(Q_k^L-Q^*)\\
 \le& (\rho+\varepsilon-2)\|Q_k^L-Q^*\|^2 + \alpha^2 \lambda_{\max}(M) {\mathbb E}[w_k^T w_k|{\cal F}_k].
\end{align*}

Taking the expectation ${\mathbb E}[\cdot]$ on both sides and rearranging terms yield
\begin{align*}
&(2-\rho-\varepsilon) {\mathbb E}[\|Q_k^L - Q^*\|^2]\\
\le& {\mathbb E}[V(Q_k^L-Q^*)] - {\mathbb E}[V(Q_{k + 1}^L - Q^*)]+ \alpha ^2 \lambda_{\max} (M) {\mathbb E}[w_k^T w_k]
\end{align*}

Next we show that the variance of $w_k$ is bounded:
\[
{\mathbb E}[w_k^T w_k |{\cal F}_k] \le W:= \frac{16|{\cal S}||{\cal A}|}{(1 - \gamma)^2}.
\]
This is because
\begin{align*}
\|w_k\|_\infty \le & \|((e_a\otimes e_s) - D)r_k\|_\infty \\
& + \gamma \|(e_a \otimes e_s)(e_{s'})^T - DP \|_\infty  \|\Pi_{Q_k}\|_\infty  \|Q_k\|_\infty \\
& + \| ((e_a\otimes e_s )(e_a\otimes e_s)^T - D)\|_\infty  \|Q_k\|_\infty\\
% \le & 2R_{\max} + \gamma \|(e_a \otimes e_s)(e_{s'})^T - DP \|_\infty  \|\Pi_{Q_k}\|_\infty  Q_{\max}\\
% &+ 2Q_{\max}\\
\le & 2R_{\max}+2\gamma Q_{\max} + 2Q_{\max}\\
\le & \frac{4}{1-\gamma},
\end{align*}
where the last inequality comes from Assumptions~\ref{assumption:bounded-reward}-\ref{assumption:bounded-Q0}, and~\cref{lemma:bounded-Q}. Hence, ${\mathbb E}[w_k^T w_k| {\cal F}_k] \le W$.

Summing both sides from $k=0$ to $k=N-1$ and dividing by $N$ and $2 - \rho  - \varepsilon >0$ leads to
\begin{align*}
&\frac{1}{N}\sum_{k=0}^{N-1} {{\mathbb E}[\|Q_k^L-Q^*\|^2]}\\
\le & \frac{\alpha^2\lambda_{\max}(M)W}{2-\rho-\varepsilon}+\frac{1}{2-\rho-\varepsilon}\frac{1}{N}{\mathbb E}[V(Q_0^L-Q^*)]\\
\le& \alpha^2 \lambda_{\max}(M)W+\frac{\lambda_{\max}(M)}{N} {\mathbb E}[\|Q_0^L-Q^*\|^2]
\end{align*}
where we used $\lambda_{\min}(M) \|x\|_2^2\le V(x) \le\lambda_{\max}(M) \|x\|_2^2$. We use the bound $\lambda_{\max}(M) \le \frac{|{\cal S}||{\cal A}|}{1-\left({\frac{\rho }{\rho+\epsilon} }\right)^2}$ in~\cref{prop:Lyapunov-theorem}, let $Q_0^L = Q_0$, and set $\varepsilon  = \frac{1-\rho}{2}$ so that $\rho+\varepsilon=\frac{1+\rho}{2}\in (0,1)$ to have
\begin{align*}
&\sum_{k=0}^{N-1} {\frac{1}{N}} {\mathbb E}[\|Q_k^L-Q^*\|^2]\\
\le & \frac{\alpha^2 |{\cal S}||{\cal A}|W}{1-\left(\frac{\rho}{\rho+\varepsilon} \right)^2} + \frac{1}{N}\frac{|{\cal S}||{\cal A}|}{1-\left(\frac{\rho}{\rho+\varepsilon} \right)^2} {\mathbb E}[\|Q_0-Q^*\|^2]\\
\leq & \alpha^2 \frac{(1+\rho)|{\cal S}||{\cal A}|W}{1-\rho} + \frac{1}{N}\frac{(1+\rho)|{\cal S}||{\cal A}|}{1-\rho} {\mathbb E}[\|Q_0-Q^*\|^2]\\
\le& \frac{2\alpha |{\cal S}||{\cal A}|W}{d_{\min}(1-\gamma)} + \frac{1}{N}\frac{2|{\cal S}||{\cal A}|}{\alpha d_{\min} (1-\gamma)} {\mathbb E}[\|Q_0-Q^*\|^2]
\end{align*}

Taking the square root on both sides, using the subadditivity of the square root, and combining with the relations
\begin{align*}
&\sqrt{\sum_{k=0}^{N-1} {\frac{1}{N} {\mathbb E}[\|Q_k^L - Q^*\|_2^2]}}\\
\ge& \sum_{k=0}^{N-1} {\frac{1}{N} {\mathbb E}[\|Q_k^L-Q^*\|_2]}\ge \sum_{k=0}^{N-1} {\frac{1}{N} {\mathbb E}[\|Q_k-Q^*\|_\infty]}
\end{align*}
and
\begin{align*}
\|Q_0 - Q^*\|_2^2\le& |{\cal S}||{\cal A}|\|Q_0-Q^*\|_\infty^2,
\end{align*}
which applies the concavity of the square root function and Jensen inequality, we further have
\begin{align}
\sum\limits_{k = 0}^{N - 1} {\frac{1}{N} {\mathbb E}[\left\| {Q_k^L  - Q^* } \right\|_\infty  ]}  \le \sqrt {\frac{{32\alpha |{\cal S}|^2 |{\cal A}|^2 }}{{d_{\min } (1 - \gamma )^3 }} + \frac{1}{N}\frac{{2|{\cal S}|^2 |{\cal A}|^2 }}{{\alpha d_{\min } (1 - \gamma )}} {\mathbb E}[\left\| {Q_0  - Q^* } \right\|_\infty ^2 ]}.\label{eq:6}
\end{align}

Using the Jensen inequality again yields the desired result.
\end{proof}

Before closing this subsection, we provide a simple example which shows the case that the gap between the lower comparison system and the original system is tight.
\begin{example}\label{ex:1}
Consider an MDP with ${\cal S}=\{1\}$, ${\cal A}=\{1\}$, $\gamma = 0.9$, where a reward is one at every time instances. In this case, the optimal policy is defined with $\pi ^*(1) = 1$, and the corresponding optimal Q-function is $Q^* = \frac{1}{{1 - \gamma }}$. The overall system is deterministic. In this case, $D=P = 1$ and $\Pi_Q=1$ for any $Q \in {\mathbb R}$. Then, we have $A_Q  = 1 + \alpha (0.9 - 1),b_Q  = 0$, and the switching system in~\eqref{eq:swithcing-system-form} is given as
\[
Q_{k + 1}  - Q^*  = (1 - 0.1\alpha )(Q_k  - Q^* ).
\]
On the other hand, since $A_{Q^*}  = 1 + \alpha (0.9 - 1)$, the lower system in~\eqref{eq:lower-system} is the same as the original system, i.e.,
\[
Q^L_{k + 1}  - Q^*  = (1 - 0.1\alpha )(Q^L_k  - Q^* ).
\]
Therefore, with $Q_0 = Q_0^L$, the lower bound is tight in the sense that $Q_k -Q^* = Q_k^L - Q^*$ for all $k \geq 0$.
\end{example}

\subsection{Upper comparison system}
Now consider the stochastic linear switching system
\begin{align}
Q_{k+1}^U-Q^*= A_{Q_k}(Q_k^U-Q^*)+\alpha w_k,\quad Q_0^U-Q^*\in {\mathbb R}^{|{\cal S}||{\cal A}|},
\label{eq:upper-system}
\end{align}
 where the stochastic noise $w_k$ is kept the same as the original system. We will call it the \emph{upper comparison system}.
% Similar to the lower comparison system, its main property is that if $Q_0^U-Q^*\ge Q_0-Q^*$ initially, then $Q_k^U-Q^*\ge Q_k-Q^*$ for all $k \ge 0$.
\begin{proposition}
Suppose $Q_0^U-Q^*\ge Q_0-Q^*$, where $\geq $ is used as the element-wise inequality. Then,
$$Q_k^U-Q^*\ge Q_k-Q^*,$$
for all $k \ge 0$.
\end{proposition}
\begin{proof}
Suppose the result holds for some $k \ge 0$. Then,
% \begin{align*}
% &(Q_{k+1}- Q^*)\\
% =& (Q_k- Q^*)\\
% & + \alpha D\{\gamma P\Pi_{Q_k} Q_k-\gamma P\Pi_{Q^*} Q^*-Q_k+Q^*\}+\alpha w_{k}\\
% =&\{I + \alpha (\gamma DP\Pi_{Q_k}-D)\}(Q_k-Q^*)\\
% &+\alpha D(\gamma P\Pi_{Q_k} Q^*-\gamma P\Pi_{Q^*} Q^*)+\alpha w_{k}\\
% \le & \{I + \alpha (\gamma DP\Pi_{Q_k}- D)\}(Q_k-Q^*)+\alpha w_k\\
% % \le& \{I + \alpha (\gamma DP\Pi_{Q_k}- D)\} (Q_k-Q^*)\\
% % &+\{I + \alpha (\gamma DP\Pi_{Q_k}-D)\}(Q_k^U- Q_k)+\alpha w_k\\
% =&\{I + \alpha (\gamma DP\Pi_{Q_k} - D)\}(Q_k^U - Q^*)+\alpha w_k\\
% =& Q_{k+1}^U-Q^*,
% \end{align*}
\begin{align*}
(Q_{k+1}- Q^*)
=& A_{Q_k}(Q_k- Q^*)+b_{Q_k}+\alpha w_{k}\\
\leq& A_{Q_k}(Q_k- Q^*)+\alpha w_{k}\\
\leq &A_{Q_k}(Q_k^U - Q^*)+\alpha w_k\\
=& Q_{k+1}^U-Q^*,
\end{align*}
where we used the fact that $b_{Q_k}=D(\gamma P\Pi_{Q_k} Q^*-\gamma P\Pi_{Q^*} Q^*)\le D(\gamma P\Pi_{Q^*} Q^*-\gamma P\Pi_{Q^*} Q^*)=0$ in first inequality. The second inequality is due to the hypothesis $Q_k^U-Q^*\ge Q_k-Q^*$ and the fact that $A_{Q_k}$ is a positive matrix. The proof is completed by induction.
\end{proof}
\begin{remark}
In the original system~\eqref{eq:swithcing-system-form}, one can easily prove that\\ $b_{Q_k}:=\gamma DP(\Pi_{Q_k} - \Pi_{Q^*})Q^* \leq 0$ using the definition of $\Pi _{Q_k }$, i.e., $\Pi _{Q_k} Q^* \leq \Pi _{Q^*} Q^*$. Intuitively, removing this nonpositive bias term, $b_{Q_k}$, leads to the upper comparison system. The vector $_{Q_k}$ represents a portion of the gap between the original and upper systems incurred at a single time-step.
\end{remark}

Hence,  the trajectory of the stochastic linear switching system bounds that of the original system from above. Note  that the system matrix $A_{Q_k}$ switches according to the change of $Q_k$, which  depends probabilistically on $Q_k^U$. Therefore, if we take the expectation on both sides, it is not possible to separate $A_{Q_k}$ and the state $Q_k^U-Q^*$ unlike the lower comparison system, making it much harder to analyze the stability of the upper comparison system.

 To circumvent such a difficulty, we instead study an error system by subtracting the lower comparison system from the upper comparison system:
\begin{align}
Q_{k+1}^U- Q_{k+1}^L
% =& \{I + \alpha (\gamma DP\Pi_{Q_k}-D)\}(Q_k^U-Q^*)\\
% & - \{I+\alpha (\gamma DP\Pi_{Q^*}-D)\} (Q_k^L-Q^*)\\
% =&\{I + \alpha (\gamma DP\Pi_{Q_k}-D)\} (Q_k^U-Q^*)\\
% &-\{I + \alpha (\gamma DP\Pi_{Q_k} - D)\}(Q_k^L-Q^*)\\
% &+ \alpha (\gamma DP\Pi_{Q_k}-\gamma DP\Pi_{Q^*})(Q_k^L-Q^*)\\
=A_{Q_k}(Q_k^U-Q_k^L) +B_{Q_k}(Q_k^L-Q^*), \label{eq:error-system}
\end{align}
where
$$B_{Q_k}:=A_{Q_k}-A_{Q^*}=\alpha \gamma DP(\Pi_{Q_k}-\Pi_{Q^*}).$$
Here the stochastic noise $\alpha w_k$ is canceled out in the error system. Matrices $(A_{Q_k}, B_{Q_k})$ switches according to the external signal $Q_k$, and $Q_k^L-Q^*$ can be seen as an external disturbance.

% In summary, we have the error system
% \begin{align}
% Q_{k+1}^U-Q_{k+1}^L=& \underbrace {\{I + \alpha (\gamma DP\Pi_{Q_k}-D)\}}_{=:A_{Q_k}}(Q_k^U-Q_k^L)\nonumber\\
% &+ \underbrace {\alpha \gamma DP(\Pi_{Q_k}-\Pi_{Q^*})}_{=:B_{Q_k}}(Q_k^L-Q^*),\label{eq:error-system}\\
% &Q_0^L- Q^*\in {\mathbb R}^{|{\cal S}||{\cal A}|}\nonumber
% \end{align}
% where $Q_k - Q_k^L\ge 0$ holds for all $k\ge 0$.
The key insight is as follows: if we can prove the stability of the error system, i.e., $Q_k^U-Q_k^L \to 0$ as $k\to\infty$, then since $Q_k^L \to Q^*$ as $k \to \infty$, we have $Q_k^U \to Q^*$ as well.

\begin{example}\label{ex:2}
Consider~\cref{ex:1} again. The upper system in~\eqref{eq:upper-system} is the same as the original system, i.e.,
\[
Q^U_{k + 1}  - Q^*  = (1 - 0.1\alpha )(Q^U_k  - Q^* ).
\]
Therefore, with $Q_0 = Q_0^U$, the upper bound is tight in the sense that $Q^U_k - Q^* = Q_k - Q^*$ for all $k\geq 0$.
\end{example}

\begin{example}
Consider an MDP with ${\cal S}=\{1\}$, ${\cal A}=\{1,2\}$, $\gamma = 0.9$, where the reward is one when $a=1$ and zero otherwise. In this case, the optimal policy is defined with $\pi ^*(1) = 1$, and the corresponding optimal Q-function is $Q^* (s,1) = \frac{1}{{1 - \gamma }} = 10,Q^* (s,2) = \frac{\gamma }{{1 - \gamma }} = 9$. We consider the behavior policy $\beta ( \cdot |1) = \left[ {\begin{array}{*{20}c}
   {0.5} & {0.5}  \\
\end{array}} \right]^T$. Then, we have
\[
Q = \begin{bmatrix}
   {Q(1,1)}  \\
   {Q(1,2)}  \\
\end{bmatrix}, R = \begin{bmatrix}
   1  \\
   0  \\
\end{bmatrix}, P = \begin{bmatrix}
   1  \\
   1  \\
\end{bmatrix}, D = \begin{bmatrix}
   {0.5} & 0  \\
   0 & {0.5}  \\
\end{bmatrix}
\]
and $\Pi _Q  = \left[ {\begin{array}{*{20}c}
   1 & 0  \\
\end{array}} \right]$ if ${Q(1,1) \ge Q(1,2)}$ and $\Pi _Q  = \left[ {\begin{array}{*{20}c}
   0 & 1  \\
\end{array}} \right]$ otherwise. In this case,
\[
A_Q  = \begin{bmatrix}
   1 & 0  \\
   0 & 1  \\
\end{bmatrix} + \alpha \left( {0.9 \begin{bmatrix}
   {0.5}  \\
   {0.5}  \\
\end{bmatrix} \Pi _Q  - \begin{bmatrix}
   {0.5} & 0  \\
   0 & {0.5}  \\
\end{bmatrix}} \right)
\]
and
\[
b_Q  = 0.9\begin{bmatrix}
   {0.5}  \\
   {0.5}  \\
\end{bmatrix} \left( \Pi _Q  - \begin{bmatrix}
   1 & 0  \\
\end{bmatrix} \right)\begin{bmatrix}
   {10}  \\
   9  \\
\end{bmatrix}
\]
From the result, the gap, $Q^U_{k+1} - Q_{k+1}$, between the upper and original systems incurred at each time step is $b_Q  = \begin{bmatrix}
   0  \\
   0  \\
\end{bmatrix}$ when $Q(1,1)\geq Q(1,2)$, and $b_Q  =  - 0.9 \begin{bmatrix}
   {0.5}  \\
   {0.5}  \\
\end{bmatrix}$ when $Q(1,1)< Q(1,2)$. Similar results can be obtained for the lower system.
\end{example}

\subsection{Finite-time error bound of Q-learning}
In this subsection, we provide a finite-time error bound of Q-learning. We obtain the following main result.
\begin{theorem}\label{prop:expected-error-bound-Q-learning}
Under Assumptions 1-4, for any $N \geq 0$, we have the following error bound for Q-learning iterates:
\begin{align}
&{\mathbb E}[\|\tilde Q_N  - Q^* \|_\infty] \le \left( {\frac{{4\gamma d_{\max }  + d_{\min } (1 - \gamma )}}{{d_{\min }^{3/2} (1 - \gamma )^{5/2} }}|{\cal S}||{\cal A}|} \right)\sqrt {32\alpha  + \frac{1}{N}\frac{4}{\alpha }}
\label{eq:3}
\end{align}
where $\alpha \in (0,1)$ is the constant step-size and $\tilde Q_N=\frac{1}{N}\sum_{i=0}^{N-1} Q_k$.
\end{theorem}
\begin{proof}
Taking the norm on both sides of the error system, we have for any $k\geq0$
\begin{align*}
&\|Q_{k+1}^U - Q_{k+1}^L\|_\infty\\
\le& \|A_{Q_k}\|_\infty \|Q_k^U-Q_k^L\|_\infty + \|B_{Q_k}\|_\infty \|Q_k^L-Q^*\|_\infty\\
\le& (\rho+\varepsilon) \|Q_k^U-Q_k^L\|_\infty + \|B_{Q_k}\|_\infty  \|Q_k^L- Q^*\|_\infty\\
\le & (\rho+\varepsilon) \|Q_k^U-Q_k^L\|_\infty + 2\alpha \gamma d_{\max} \|Q_k^L-Q^*\|_\infty.
\end{align*}
Here the last inequality uses the fact that
\begin{align*}
\|B_{Q_k}\|_\infty\le& \alpha\gamma d_{\max} \|P(\Pi_{Q_k}-\Pi_{Q^*})\|_\infty
\le 2\alpha\gamma d_{\max}.
\end{align*}

Rearranging terms leads to
\begin{align*}
&(1-\rho-\varepsilon) \|Q_k^U-Q_k^L\|_\infty\\
\le& \|Q_k^U-Q_k^L\|_\infty - \|Q_{k+1}^U-Q_{k+1}^L\|_\infty + 2\alpha \gamma d_{\max} \|Q_k^L-Q^*\|_\infty,\quad k \ge 0.
\end{align*}

Summing both sides from $k=0$ to $k=N-1$, dividing by $N$ and $2-\rho-\varepsilon>0$, and letting $Q_0^U= Q_0^L=Q_0$ lead to
\begin{align}
\sum_{k=0}^{N-1} {\frac{1}{N} \|Q_k^U-Q_k^L\|_\infty} \le \frac{2\alpha\gamma d_{\max}}{1-\rho-\varepsilon}\sum_{k=0}^{N-1} {\frac{1}{N}\|Q_k^L-Q^*\|_\infty}\label{eq:5}
\end{align}

Next, we will express the left-hand side in terms of $Q_k$. By triangle inequality, we have
\begin{align*}
\|Q_k-Q^*\|_\infty \le & \|Q^*- Q_k^L\|_\infty +\|Q_k-Q_k^L\|_\infty\\
\le & \|Q^*- Q_k^L\|_\infty +\|Q_k^U-Q_k^L\|_\infty.
\end{align*}
The second inequality is because
$$0 \le Q_k-Q_k^L\le Q_k^U-Q_k^L.$$
This leads to
\[
\|Q_k-Q^*\|_\infty - \|Q^*-Q_k^L\|_\infty \le \|Q_k^U-Q_k^L\|_\infty
\]

Combining this inequality with~\eqref{eq:5}, one gets
\begin{align*}
&\sum_{k=0}^{N-1} {\frac{1}{N}(\|Q_k-Q^*\|_\infty- \|Q^*-Q_k^L\|_\infty)}\\
\le& \frac{4\gamma d_{\max}}{d_{\min}(1 -\gamma)}\sum_{k=0}^{N-1} {\frac{1}{N}\|Q_k^L-Q^*\|_\infty}
\end{align*}
where we let $\varepsilon=\frac{1-\rho}{2}$ so that $\rho+\varepsilon= \frac{1+\rho}{2}$ and $\frac{1}{1-\rho-\varepsilon} = \frac{2}{1-\rho} = \frac{2}{\alpha d_{\min}(1-\gamma)}$. Rearranging terms again, taking the expectation on both sides, and combining it with~\eqref{eq:6}, we obtain
\begin{align*}
& \sum_{k=0}^{N-1}{\frac{1}{N}{\mathbb E}[\|Q_k-Q^*\|_\infty]}\\
\le& \frac{5d_{\max}}{d_{\min}(1-\gamma)}\sum_{k = 0}^{N-1} {\frac{1}{N} {\mathbb E}[\|Q_k^L-Q^*\|_\infty]}\\
\le & \frac{{5d_{\max } }}{{d_{\min } (1 - \gamma )}}\sqrt {\frac{{32\alpha |{\cal S}|^2 |{\cal A}|^2 }}{{d_{\min } (1 - \gamma )^3 }} + \frac{1}{N}\frac{{2|{\cal S}|^2 |{\cal A}|^2 }}{{\alpha d_{\min } (1 - \gamma )}} {\mathbb E}[\left\| {Q_0  - Q^* } \right\|_\infty ^2 ]}.
\end{align*}

From~\cref{lemma:bounded-Q}, $-{\bf 1}Q_{\max}\le Q_k \le {\bf 1}Q_{\max}$ holds for all $k\ge 0$. Applying $Q_{\max}= 1/(1-\gamma)$ and $\|Q_0\|_\infty\le 1$, the Jensen inequality, and after simplifications, we can obtain the desired conclusion.
\end{proof}
\begin{remark}
Lyapunov theory~\cite{chen2021lyapunov} has been applied for the lower comparison system, which is a linear time-invariant system. On the other hand, the techniques used for the error system between the upper and lower comparison systems more resemble those used in the optimization community rather than leveraging the nature of switching dynamical system. However, the switching system formulation captures essential behaviors of Q-learning algorithm, and itself is mainly used in combination with the lower comparison system in the overall derivation process.
\end{remark}

\subsection{Remarks}

\emph{Overestimation and maximization bias.}
% Overall, the proposed analysis based on switching system models proves that Q-learning with a constant step-size may converge exponentially with a constant error term.
Our analysis provides an intuitive explanation of the well-known overestimation phenomenon in Q-learning~\cite{hasselt2010double}:
% At each iteration $k$, $Q_k(s,a)$ can be viewed as a random variable estimating $Q^*(s,a)$, i.e., $Q_k(s,a) = Q^*(s,a) + e_k(s,a)$, where $e_k(s,a) \in {\mathbb R}$ is an estimation error.
% The overestimation phenomenon usually indicates the fact that
$Q_k(s,a)$ tends to overestimate $Q^*(s,a)$ due to the maximization bias in the Q-learning updates. This becomes severe especially when the action-space is large. In particular, it can be problematic when the action spaces depending on states are heterogeneous  and the current estimate $Q_k$ is used for the exploration, e.g., the $\varepsilon$-greedy behavior policy. In this case, since $\argmax_{a\in {\cal A}}Q_k (s,a)$ tends to choose actions with larger maximization biases, thus degrading the quality of exploration and leading to slower convergence. Moreover, the overestimation error could be amplified at each iteration $k$ when it passes through the max operator.

In fact, assuming that the initial $Q_0(s,a)- Q^*(s,a)$ is a zero mean random variable, namely, ${\mathbb E}[Q_0 - Q^*] = 0$, we can easily see through our analysis that
\[{\mathbb E}[Q_k - Q^*]\geq 0, \quad \forall k\geq0.\]
This is because the lower comparison system (which is a stochastic linear system) satisfies that
\[
{\mathbb E}[Q_k^L-Q^*] = A_{Q^*}^k {\mathbb E}[Q_0^L-Q^*] + \sum_{i=0}^{k-1} {A_{Q^*}^{k-1-i}\alpha {\mathbb E}[w_i]}=0.
\]
provided that $Q_0^L = Q_0$, namely, there exists no biases in the lower system state. On the other hand, since $Q_k^L-Q^*\leq Q_k - Q^*$, ${\mathbb E}[Q_k - Q^*]\geq0$ holds. For $(s,a)$ such that $Q_k^L\leq Q_k$ holds strictly, then ${\mathbb E}[Q_k(s,a)]>Q^*(s,a)$, which explains the overestimation phenomenon.

\subsection{Sample complexity}

Based on the finite-time error bound on the Q-learning iterates in~\cref{prop:expected-error-bound-Q-learning},
we can derive an upper bound on the sample or iteration complexity of Q-learning: to find an $\varepsilon$-optimal solution such that ${\mathbb E}[\|\tilde Q_N-Q^*\|_\infty]<\varepsilon$, we need  at most
\begin{align*}
{\cal O}\left( \frac{d_{\max}^4 |{\cal S}|^4 |{\cal A}|^4}{\varepsilon^4\delta^4 d_{\min}^6 (1-\gamma)^{10}} \right)
\end{align*}
samples. If the state-action pair is sampled uniformly from ${\cal S}\times {\cal A}$, then $d(s,a) = \frac{1}{|{\cal S}||{\cal A}|},\forall (s,a) \in {\cal S} \times {\cal A}$ and $d_{\min}=d_{\max}=\frac{1}{|{\cal S}||{\cal A}|}$.
In this case, the sample complexity becomes ${\cal O}\left(\frac{|{\cal S}|^6 |{\cal A}|^6}{\varepsilon^4 \delta^4 (1-\gamma)^{10}} \right)$. The proof is given in Appendix~\ref{app:sample_complexity}.

The finite-time analysis of asynchronous Q-learning with constant step-size was first considered in~\cite{beck2012error} and has been recently studied in~\cite{li2020sample} and the concurrent work~\cite{chen2021lyapunov}. Based on the cover time assumption,~\cite{beck2012error} provides $\tilde {\cal O}\left( {\frac{t_{\rm cover}^3 |{\cal S}||{\cal A}|}{(1-\gamma)^5 \varepsilon^2}} \right)$, where $t_{\rm cover}$ is the cover time and $\tilde {\cal O}$ ignores the polylogarithmic factors. The results in~\cite{li2020sample} provide $\tilde {\cal O}\left( \frac{1}{d_{\min}(1-\gamma)^5 \varepsilon^2} + \frac{t_{\rm mix}}{d_{\min}(1-\gamma)}\right)$ with a single Markovian trajectory, where $t_{\rm mix}$ is the mixing time. Note that the mixing time and cover time assumptions are adopted in~\cite{li2020sample}. The complexity $\tilde {\cal O}\left(\frac{1}{d_{\min}^3(1-\gamma)^5\varepsilon^2}\right)$ is given in~\cite{chen2021lyapunov} with a single Markovian trajectory. Note that the bounds in~\cite{chen2021lyapunov} and~\cite{beck2012error} are for the expected error bounds, and those in~\cite{li2020sample} are for the concentration error bounds. Besides,~\cite{qu2020finite} offers a sharper bound using a diminishing step-size. Based on the analysis, we summarize advantages and limitations of the proposed approach. A limitation of the proposed method lies in that the corresponding sample complexity is not tighter than the existing approaches. On the other hand, the main advantage is the proposition of unique switching system and control perspectives, which inherit simplicity and elegance. It also provides additional insights on Q-learning.

\section{Example}\label{sec:example}

Consider an MDP with ${\cal S}=\{1,2\}$, ${\cal A}=\{1,2\}$, $\gamma = 0.9$,
\begin{align*}
%================================================================
&P_1=\begin{bmatrix}
   0.3863 & 0.6137\\
   0.3604 & 0.6396 \\
\end{bmatrix},\quad P_2=\begin{bmatrix}
   0.8639 & 0.1361 \\
   0.7971 & 0.2029\\
\end{bmatrix},
%================================================================
\end{align*}
and the reward function
\[
r(1,1)=-3,\quad r(2,1)=1,\quad r(1,2)=2,\quad r(2,2)=-1.
\]
Actions are sampled using the behavior policy $\beta$
\begin{align*}
&{\mathbb P}[a_k = 1|s_k = 1] = 0.2,\quad {\mathbb P}[a_k = 2|s_k = 1] = 0.8,\\
&{\mathbb P}[a_k = 1|s_k = 2] = 0.7,\quad {\mathbb P}[a_k = 2|s_k = 2] = 0.3,
\end{align*}
and the states are sampled according to the distribution
\[
{\mathbb P}[s_k=1] = 0.2,\quad {\mathbb P}[s_k=2]=0.8,\quad \forall k \ge 0.
\]

Simulated trajectories of the switching system model of Q-learning (black solid line) with $\alpha = 0.002$, the lower comparison system (blue solid line), and the upper comparison system (red solid line) are depicted in~\cref{fig:1}. It demonstrates that the state of the switching model of Q-learning, $Q_k-Q^*$, is underestimated by the lower comparison system's state $Q_k^L-Q^*$ and overestimated by the upper comparison system's state $Q_k^U-Q^*$.

The evolution of the error between the upper and lower comparison systems are depicted in~\cref{fig:2}. It also demonstrates that the state of the error, $Q_k^U-Q_k^L$, converges to the origin. The simulation study empirically proves that the bounding rules that we predicted theoretically hold.

\begin{figure*}[ht]
%=============================================================================================================
\centering\includegraphics[width=17cm,height=11cm]{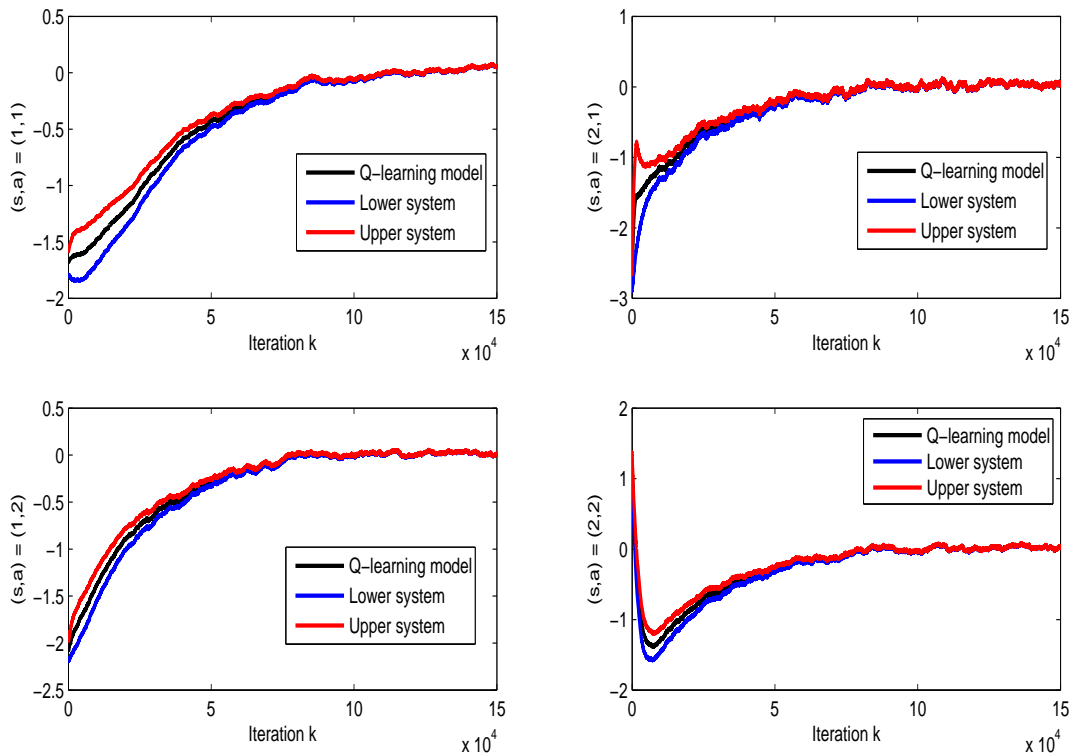}
\caption{Evolution of $Q_k-Q^*$ (black solid lines), lower comparison system $Q_k^L-Q^*$ (blue solid lines), and upper comparison system $Q_k^U-Q^*$ (red solid lines) with step-size $\alpha = 0.002$. }\label{fig:1}
%=============================================================================================================
\end{figure*}
\begin{figure*}[ht]
%=============================================================================================================
\centering\includegraphics[width=17cm,height=11cm]{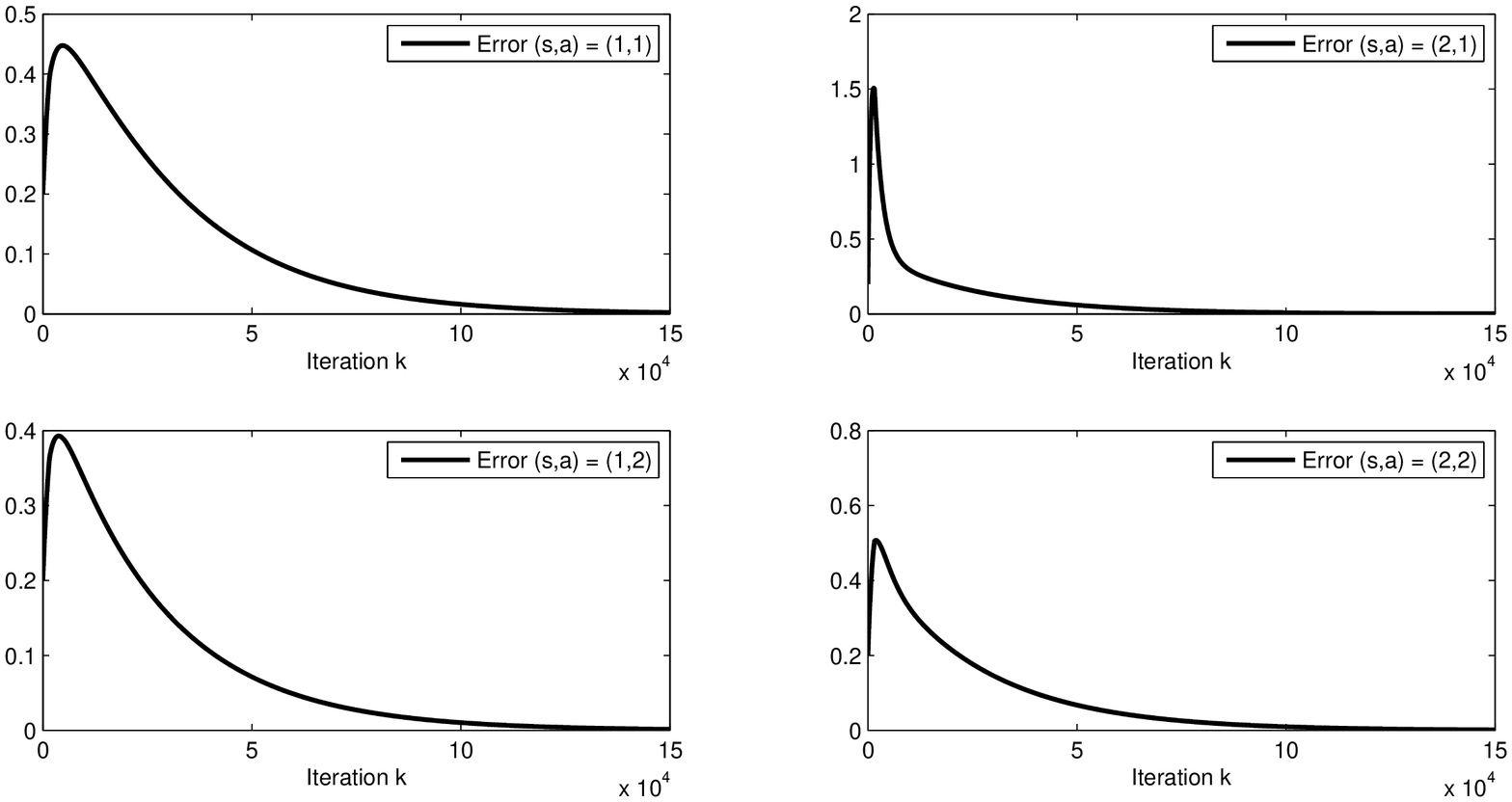}
\caption{Evolution of error $Q_k^U-Q_k^L$ (black solid lines) with step-size $\alpha = 0.002$. }\label{fig:2}
%=============================================================================================================
\end{figure*}

Under the same conditions, the simulation results with the step-size $\alpha = 0.9$ are given in~\cref{fig:3} and~\cref{fig:4}.
\cref{fig:3} shows that the evolution of $Q_k-Q^*$ (black solid lines), lower comparison system $Q_k^L-Q^*$ (blue solid lines), and upper comparison system $Q_k^U-Q^*$ (red solid lines) are less stable with a larger step-size. \cref{fig:4} also shows large fluctuation of the error $Q_k^U-Q_k^L$ with $\alpha = 0.9$.

\begin{figure*}[t]
%=============================================================================================================
\centering\includegraphics[width=17cm,height=11cm]{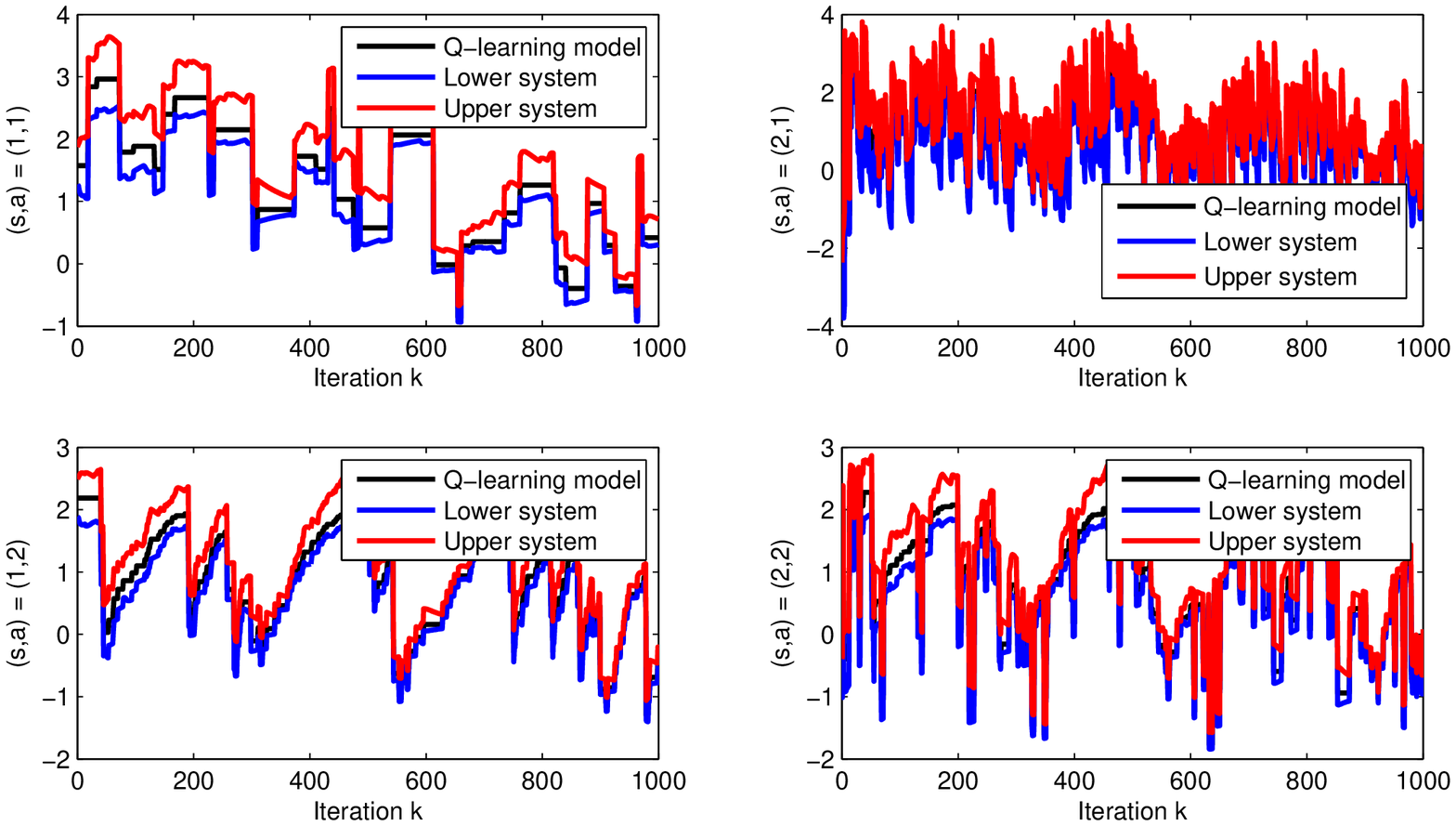}
\caption{Evolution of $Q_k-Q^*$ (black solid lines), lower comparison system $Q_k^L-Q^*$ (blue solid lines), and upper comparison system $Q_k^U-Q^*$ (red solid lines) with step-size $\alpha = 0.9$. }\label{fig:3}
%=============================================================================================================
\end{figure*}
\begin{figure*}[t]
%=============================================================================================================
\centering\includegraphics[width=17cm,height=11cm]{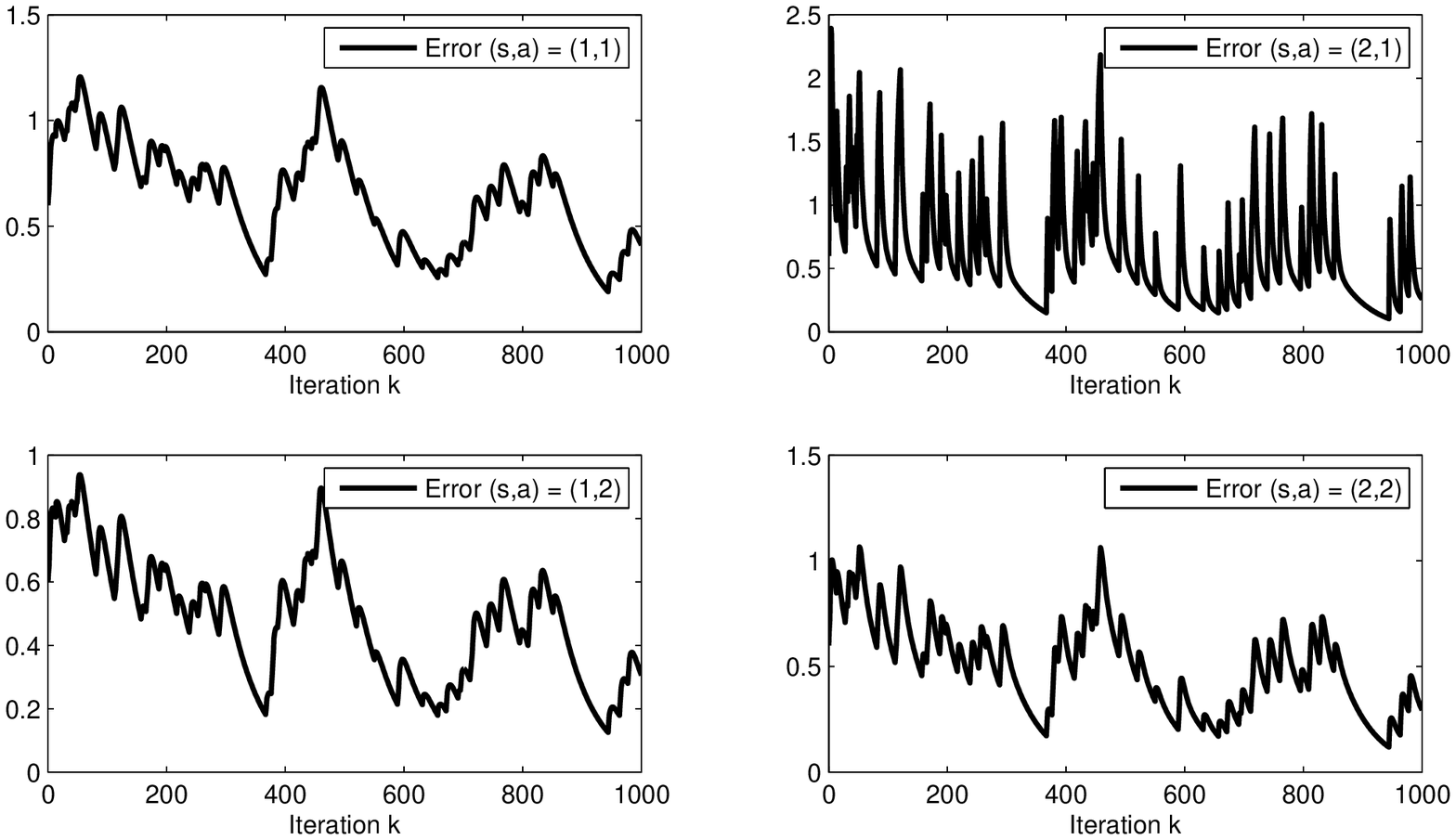}
\caption{Evolution of error $Q_k^U-Q_k^L$ (black solid lines) with step-size $\alpha = 0.9$. }\label{fig:4}
%=============================================================================================================
\end{figure*}

\section{Conclusion}\label{sec:conclusion}
In this paper, we introduced a novel control-theoretic framework based on  discrete-time switching systems to derive finite-time error bounds of Q-learning algorithm. By sandwiching the dynamics of asynchronous Q-learning with two simpler stochastic (switched) linear systems, a new finite-time analysis of the Q-learning can be easily derived. We believe it is important to emphasize that the proposed control-theoretic analysis can be viewed as a new analysis which gives additional insights on Q-learning rather than a replacement or improvement of existing convergence rate analysis. The proposed analysis has simplicity, novelty, and more intuition. We expect that, such control-theoretic analysis could further stimulate the synergy between control theory and reinforcement learning, and open up opportunities to the design of new reinforcement learning algorithms and refined analysis for Q-learning. Moreover, our approach based on the comparison systems could be of independent interest to the finite-time stability analysis of more general switching systems.

As a promising next step, we plan to further tighten the analysis and extend it to more general Markovian settings.
% Moreover, by filling the gap between the two domains in a natural and synergistic way, it offers novel insights on Q-learning and can potentially enrich both fields.
Our analysis framework can potentially be applied to derive finite-time error bounds for other variants of Q-learning, such as the double Q-learning~\cite{hasselt2010double}, averaging Q-learning~\cite{lee2020unified}, speedy Q-learning~\cite{azar2011speedy}, and multi-agent Q-learning~\cite{kar2013cal}, as well as their function approximation counterparts. We will leave this for future investigation.

% The analysis can be applied to both the i.i.d. setting and Markovian setting if the initial state distribution is already stationary under a behavior policy. However, it can be generalized to the convergence with a arbitrary initial state distribution under a mixing time assumption. The extension may require substantial further work, and it is remained as a potential future topic.

\bibliographystyle{IEEEtran}
\bibliography{reference}

% Generated by IEEEtran.bst, version: 1.14 (2015/08/26)
\begin{thebibliography}{10}
\providecommand{\url}[1]{#1}
\csname url@samestyle\endcsname
\providecommand{\newblock}{\relax}
\providecommand{\bibinfo}[2]{#2}
\providecommand{\BIBentrySTDinterwordspacing}{\spaceskip=0pt\relax}
\providecommand{\BIBentryALTinterwordstretchfactor}{4}
\providecommand{\BIBentryALTinterwordspacing}{\spaceskip=\fontdimen2\font plus
\BIBentryALTinterwordstretchfactor\fontdimen3\font minus
  \fontdimen4\font\relax}
\providecommand{\BIBforeignlanguage}[2]{{%
\expandafter\ifx\csname l@#1\endcsname\relax
\typeout{** WARNING: IEEEtran.bst: No hyphenation pattern has been}%
\typeout{** loaded for the language `#1'. Using the pattern for}%
\typeout{** the default language instead.}%
\else
\language=\csname l@#1\endcsname
\fi
#2}}
\providecommand{\BIBdecl}{\relax}
\BIBdecl

\bibitem{watkins1992q}
C.~J. C.~H. Watkins and P.~Dayan, ``Q-learning,'' \emph{Machine learning},
  vol.~8, no. 3-4, pp. 279--292, 1992.

\bibitem{tsitsiklis1994asynchronous}
J.~N. Tsitsiklis, ``Asynchronous stochastic approximation and q-learning,''
  \emph{Machine learning}, vol.~16, no.~3, pp. 185--202, 1994.

\bibitem{jaakkola1994convergence}
T.~Jaakkola, M.~I. Jordan, and S.~P. Singh, ``Convergence of stochastic
  iterative dynamic programming algorithms,'' in \emph{Advances in neural
  information processing systems}, 1994, pp. 703--710.

\bibitem{borkar2000ode}
V.~S. Borkar and S.~P. Meyn, ``The {ODE} method for convergence of stochastic
  approximation and reinforcement learning,'' \emph{SIAM Journal on Control and
  Optimization}, vol.~38, no.~2, pp. 447--469, 2000.

\bibitem{szepesvari1998asymptotic}
C.~Szepesv{\'a}ri, ``The asymptotic convergence-rate of {Q}-learning,'' in
  \emph{Advances in Neural Information Processing Systems}, 1998, pp.
  1064--1070.

\bibitem{kearns1999finite}
M.~J. Kearns and S.~P. Singh, ``Finite-sample convergence rates for
  {Q}-learning and indirect algorithms,'' in \emph{Advances in neural
  information processing systems}, 1999, pp. 996--1002.

\bibitem{even2003learning}
E.~Even-Dar and Y.~Mansour, ``Learning rates for {Q}-learning,'' \emph{Journal
  of machine learning Research}, vol.~5, no. Dec, pp. 1--25, 2003.

\bibitem{azar2011speedy}
M.~G. Azar, R.~Munos, M.~Ghavamzadeh, and H.~J. Kappen, ``Speedy
  {Q}-learning,'' in \emph{Proceedings of the 24th International Conference on
  Neural Information Processing Systems}, 2011, pp. 2411--2419.

\bibitem{beck2012error}
C.~L. Beck and R.~Srikant, ``Error bounds for constant step-size
  {Q}-learning,'' \emph{Systems \& Control letters}, vol.~61, no.~12, pp.
  1203--1208, 2012.

\bibitem{wainwright2019stochastic}
M.~J. Wainwright, ``Stochastic approximation with cone-contractive operators:
  Sharp $\ell_\infty$-bounds for {Q}-learning,'' \emph{arXiv preprint
  arXiv:1905.06265}, 2019.

\bibitem{qu2020finite}
G.~Qu and A.~Wierman, ``Finite-time analysis of asynchronous stochastic
  approximation and {Q}-learning,'' \emph{arXiv preprint arXiv:2002.00260},
  2020.

\bibitem{li2020sample}
G.~Li, Y.~Wei, Y.~Chi, Y.~Gu, and Y.~Chen, ``Sample complexity of asynchronous
  {Q}-learning: Sharper analysis and variance reduction,'' \emph{arXiv preprint
  arXiv:2006.03041}, 2020.

\bibitem{chen2021lyapunov}
Z.~Chen, S.~T. Maguluri, S.~Shakkottai, and K.~Shanmugam, ``A {L}yapunov theory
  for finite-sample guarantees of asynchronous {Q}-learning and {TD}-learning
  variants,'' \emph{arXiv preprint arXiv:2102.01567}, 2021.

\bibitem{lee2020unified}
D.~Lee and N.~He, ``A unified switching system perspective and convergence
  analysis of q-learning algorithms,'' in \emph{34th Conference on Neural
  Information Processing Systems, NeurIPS 2020}, 2020.

\bibitem{lin2009stability}
H.~Lin and P.~J. Antsaklis, ``Stability and stabilizability of switched linear
  systems: a survey of recent results,'' \emph{IEEE Transactions on Automatic
  control}, vol.~54, no.~2, pp. 308--322, 2009.

\bibitem{hasselt2010double}
H.~V. Hasselt, ``Double {Q}-learning,'' in \emph{Advances in Neural Information
  Processing Systems}, 2010, pp. 2613--2621.

\bibitem{kar2013cal}
S.~Kar, J.~M. Moura, and H.~V. Poor, ``{QD}-learning: a collaborative
  distributed strategy for multi-agent reinforcement learning through consensus
  + innovations,'' \emph{IEEE Transactions on Signal Processing}, vol.~61,
  no.~7, pp. 1848--1862, 2013.

\bibitem{khalil2002nonlinear}
H.~K. Khalil, ``Nonlinear systems,'' \emph{Upper Saddle River}, 2002.

\bibitem{sutton1998reinforcement}
R.~S. Sutton and A.~G. Barto, \emph{Reinforcement learning: {A}n
  introduction}.\hskip 1em plus 0.5em minus 0.4em\relax MIT Press, 1998.

\bibitem{gosavi2006boundedness}
A.~Gosavi, ``Boundedness of iterates in {Q}-learning,'' \emph{Systems \&
  Control letters}, vol.~55, no.~4, pp. 347--349, 2006.

\end{thebibliography}

\section{Appendix}

\subsection{Proof of Proposition~\ref{prop:Lyapunov-theorem}}\label{app:prop-Lyapunov}

\begin{proof}
For simplicity, denote $A=A_{Q^*}$. Consider matrix $M$ such that
\begin{align}
M = \sum_{k=0}^\infty {\left({\frac{1}{\rho+\epsilon} }\right)^{2k} (A^k)^T A^k}. \label{eq:4}
\end{align}
Noting that
\begin{align*}
(\rho+\epsilon)^{-2} A^T M A + I=& \frac{1}{(\rho+\epsilon)^2}A^T \left(\sum_{k=0}^\infty {\left( {\frac{1}{\rho+\epsilon}} \right)^{2k} (A^k)^T A^k }\right)A + I\\
=& M,
\end{align*}
we have
\[
(\rho+\epsilon)^{-2} A^T M A + I = M.
\]
resulting in the desired conclusion. Next, it remains to prove the existence of $M$ by proving its boundedness. Taking the norm on $M$ leads to
\begin{align*}
\left\| P \right\|_2  =& \left\| {I + (\rho  + \varepsilon )^{ - 2} A^T A + (\rho  + \varepsilon )^{ - 4} (A^2 )^T A^2  +  \cdots } \right\|_2\\
\le& \left\| I \right\|_2  + (\rho  + \varepsilon )^{ - 2} \left\| {A^T A} \right\|_2  + (\rho  + \varepsilon )^{ - 4} \left\| {(A^2 )^T A^2 } \right\|_2  +  \cdots\\
=& \left\| I \right\|_2  + (\rho  + \varepsilon )^{ - 2} \left\| A \right\|_2^2  + (\rho  + \varepsilon )^{ - 4} \left\| {A^2 } \right\|_2^2  +  \cdots\\
=& 1 + |{\cal S}||{\cal A}|(\rho  + \varepsilon )^{ - 2} \left\| A \right\|_\infty ^2  + |S||A|(\rho  + \varepsilon )^{ - 4} \left\| {A^2 } \right\|_\infty ^2  +  \cdots\\
=& 1 - |{\cal S}||{\cal A}| + \frac{{|{\cal S}||{\cal A}|}}{{1 - \left( {\frac{\rho }{{\rho  + \varepsilon }}} \right)^2 }}.
\end{align*}

Finally, we prove the bounds on the maximum and minimum eigenvalues.
From the definition~\eqref{eq:4}, $M \succeq I$, and hence $\lambda_{\min}(M)\ge 1$. On the other hand, one gets
\begin{align*}
\lambda_{\max}(M)=& \lambda_{\max}(I + (\rho+\epsilon)^{-2} A^T A\\
&+ (\rho+\epsilon)^{-4}(A^2)^T A^2+\cdots)\\
\le& \lambda_{\max}(I) + (\rho+\epsilon)^{-2} \lambda_{\max}(A^T A)\\
&+ (\rho+\epsilon)^{-4}\lambda_{\max}((A^2)^T A^2 )+\cdots\\
=& \lambda_{\max}(I) + (\rho+\epsilon)^{-2} \|A\|_2^2 + (\rho+\epsilon)^{-4} \|A^2\|_2^2  +  \cdots\\
\le& 1 + |{\cal S}||{\cal A}|(\rho+\epsilon)^{-2} \|A\|_\infty^2 \\
& + |{\cal S}||{\cal A}|(\rho+\epsilon)^{-4} \|A^2\|_\infty ^2 + \cdots\\
\le& \frac{|{\cal S}||{\cal A}|}{1 - \left(\frac{\rho}{\rho+\epsilon} \right)^2}.
\end{align*}
The proof is completed.
\end{proof}

\subsection{Sample Complexity}\label{app:sample_complexity}

\begin{proposition}[Sample complexity]
To achieve
\[
\|\tilde Q_N-Q^*\|_\infty<\varepsilon
\]
with probability at least $1-\delta$, we need the number of samples/iterations at most
\[
{\cal O}\left( \frac{d_{\max}^4 |{\cal S}|^4 |{\cal A}|^4}{\varepsilon^4 \delta^4 d_{\min}^6 (1-\gamma)^{10}} \right)
\]
\end{proposition}
\begin{proof}
For convenience, we first find a simplified overestimate on the right-hand side of~\eqref{eq:3} as
\begin{align*}
&{\mathbb E}[\|\tilde Q_N -Q^*\|_\infty]\nonumber\\
\le& \frac{20d_{\max} |{\cal S}||{\cal A}|}{d_{\min}(1-\gamma)^2}\left( \sqrt{\frac{2\alpha}{d_{\min}(1-\gamma)}} + \sqrt{\frac{1}{N}\frac{2}{\alpha d_{\min}(1-\gamma)}} \right) = :C
\end{align*}
Applying the Markov inequality
\[
{\mathbb P}[\|\tilde Q_N-Q^*\|_\infty\ge \varepsilon] \le \frac{C}{\varepsilon },
\]
we conclude that $\| \tilde Q_N  - Q^* \|_\infty < \varepsilon$ with probability at least $1-\delta$, i.e.,
\begin{align*}
{\mathbb P}[\|\tilde Q_N - Q^*\|_\infty < \varepsilon ]\ge 1 - \delta,
\end{align*}
where
\begin{align*}
\delta  = \frac{1}{\varepsilon }\frac{{20d_{\max } |S||A|}}{{d_{\min } (1 - \gamma )^2 }}\left( {\sqrt {\frac{{2\alpha }}{{d_{\min } (1 - \gamma )}}}  + \sqrt {\frac{1}{N}\frac{2}{{\alpha d_{\min } (1 - \gamma )}}} } \right)
\end{align*}
$N$, and $\alpha$ are appropriately chosen so that $\delta\in (0,1)$. One concludes that to satisfy $\|Q_N-Q^*\|_\infty<\varepsilon$ with probability at least $1-\delta$, we should have
\begin{align*}
\delta  \ge& \underbrace {\frac{1}{\varepsilon}\frac{20d_{\max} |{\cal S}||{\cal A}|}{d_{\min} (1-\gamma)^2}\sqrt \frac{2\alpha}{d_{\min}(1-\gamma)}}_{\Phi_1} + \underbrace {\frac{1}{\varepsilon}\frac{20d_{\max} |{\cal S}||{\cal A}|}{d_{\min}(1-\gamma)^2}\sqrt \frac{1}{N}\frac{2}{\alpha d_{\min}(1-\gamma)}}_{\Phi_2}
\end{align*}
which is achieved if $\delta/2\ge\Phi_1$ and $\delta/2\ge\Phi_2$.

The first inequality is satisfied if
\begin{align}
\alpha  = \frac{\delta^2\varepsilon^2}{8}\frac{d_{\min}^3 (1-\gamma)^5}{400d_{\max }^2 |{\cal S}|^2 |{cal A}|^2}
\label{eq:7}
\end{align}
and the second inequality holds if
\begin{align*}
N \ge \frac{3200d_{\max}^2 |{\cal S}|^2 |{\cal A}|^2}{\alpha\varepsilon^2 \delta^2 d_{\min}^3 (1-\gamma)^5}
\end{align*}

Plugging~\eqref{eq:7} into the last inequality, we can arrive at the desired conclusion.
\end{proof}

\end{document}